\documentclass{article}
\usepackage[utf8]{inputenc}

\newcounter{rmk}

\usepackage{amsmath}
\usepackage{float}

\usepackage{tikz}
\usetikzlibrary{shapes, arrows}

\usepackage{comment}
\usepackage[linktocpage=true]{hyperref}
\usepackage[margin=1.2 in, top=1 in, bottom= 1.2 in]{geometry}

\usepackage{indentfirst} 
\usepackage[english]{babel}
\usepackage{amsfonts}
\usepackage{mathrsfs}
\usepackage[mathscr]{euscript}
\usepackage{bbm}
\usepackage{latexsym}
\usepackage{mathdots}
\usepackage{amssymb}
\usepackage{mathtools}
\usepackage{graphicx}
\usepackage{tikz}
\usepackage{placeins}
\usepackage{esint}

\usepackage{enumitem}
\setlist{nolistsep}
\usepackage{amsthm}
\usepackage{relsize}
\usepackage{nicefrac}
\usepackage[capitalize]{cleveref}

\newtheoremstyle{plain}{3mm}{3mm}{\slshape}{}{\bfseries}{.}{.5em}{}
\newtheoremstyle{definition}{2mm}{2mm}{}{}{\bfseries}{.}{.5em}{}
\theoremstyle{plain}
	
\newtheorem{theorem}{Theorem}

\newtheorem{lemma}[theorem]{Lemma}

\theoremstyle{definition}

\newtheorem{remark}[rmk]{Remark}

\theoremstyle{plain}
\newtheorem*{namedthm}{\namedthmname}
\newcounter{namedthm}
	

\newcommand{\R}{\mathbb{R}}

\newcommand{\D}{\mathcal{D}}

\newcommand{\eps}{\epsilon}

\newcommand{\A}{\mathcal{A}}
\newcommand{\B}{\mathcal{B}}
\newcommand{\G}{\mathcal{G}}

\title{Dirichlet Green's functions with singular drifts at the boundary of convex domains}
\author{Aritro Pathak}

\begin{document}

\maketitle

\begin{abstract}
In convex bounded domains, in $\R^{n}$ with $n\geq 3$, we establish interior pointwise upper bounds for the Dirichlet Green’s function of elliptic operators whose principal part is the Laplacian and which include a drift term that diverges near the boundary like a negative power of the distance with exponent strictly less than $1$. This work extends an earlier result for operators with such drifts in the unit ball, and streamlines the proof in particular to adapt it to the question in convex domains.
\end{abstract}

\section{Introduction}

In this paper, we consider $K\in\R^{n}$ with $n\geq 3$, a convex bounded domain, with a specified interior point which we call the origin $0$, so that the ball $B_0:=B(0, R)\in \text{int}(K)$. We denote the ball $$B_1:=B(0,\frac{1}{2}R).$$ 

We also denote, $\delta(X)=\text{dist}(X,\partial K)$, the distance to the boundary $\partial K$. The coordinates in $\R^{n}$ are written as $(x_1,\dots,x_n)$. In this paper, we extend the result of \cite{Pat25} where the domain was the unit ball. We consider in $\R^{n}$ the linear second order operator with a singular drift term, given for some $0<\beta<1$, 
\begin{equation}\label{Laplacian}
    L u=-\Delta u-\B\cdot \nabla u =0.
\end{equation}

We also have, 

 \begin{align}\label{driftt}
     |\B(X)|\delta(X)^{1-\beta}\leq M.
\end{align}

 Here $0< M<\infty$.

 As in \cite{Pat25}, we require the pointwise condition that, 
     \begin{align}\label{nonpositivee}
         -\frac{M^2}{\delta(x)^{2-2\beta}}\leq \nabla\cdot \B\leq 0, \ 
     \end{align}
     
     When the drift satisfies the pointwise estimate of \cref{nonpositivee}, then according to Theorem 1.2 of \cite{Ha24}, solutions to the adjoint equation $L^{T} u=-\nabla\cdot(\nabla u-\B u)=0$ exists, and one considers the Green's functions $\G(x,y),\G_T (x,y)$  corresponding to $L$ and $L_T$ respectively. Further, we then also have $\G(x,y)=\G_T(y,x)$. 

Here we also outline an alternate argument which ultimately uses the Fredholm alternative, to show existence of solutions for both the operator as well as the adjoint operator. 

More generally, we can define the operators
\begin{equation}\label{imp22}
    L_1 u=-\nabla\cdot(\A^T\nabla u) -\B\cdot \nabla u =0,\ \ L_{1}^{T} u=-\nabla\cdot(\A\nabla u -\B u),  \ \   
\end{equation}

where we have constants $\infty> \Lambda\geq \lambda>0$ , and $M$ as before, so that, 
\begin{align}\label{imp23}
    \A_{ij}(X) \eta_i \eta_j \geq \lambda |\eta|^2, |\A_{ij}(X)\eta_i \psi_j|\leq \Lambda |\eta| |\psi|, \ \ \text{for all} \ \eta,\psi\in \R^{n}\setminus \{0\}, \ \ \text{and} \ \delta(X)^{1-\beta}|\B(X)|\leq M.
\end{align}

Here we state the main result of this paper:

\begin{theorem}\label{thm2}
    Consider the operator of \cref{Laplacian} and the Dirichlet Green's function corresponding to this operator. Further suppose that $\B\in C^{0,\alpha}(\Omega')$ for some $\alpha>0$, for any compactly supported subdomain  $\Omega'\subset K$, and that $\B$ satisfies  \cref{driftt,nonpositivee}.  For any $x\in K$ with $|x|\leq \frac{R}{2}$, we have for some constant $T'$ dependent on $M$ so that,
    \begin{equation}
        \G(x,0)\leq T'(M)\frac{1}{|x|^{n-2}}.
    \end{equation}
\end{theorem}

The following two results are recalled from \cite{Pat25}, and the proofs remain identical in the present setting; the proof of \cref{thm3} follows by considering the adjoint Green's function along with repeated use of the Harnack inequality, once \cref{thm1} is established.

\begin{theorem}\label{thm1}
    For the elliptic operator $L_1$ in \cref{imp22} with the coefficients being only measurable, satisfying \cref{imp23}, we have the bound for the Green's function for the operator in \cref{imp22}: for any $z,y\in K$ with $|z-y|\leq \frac{1}{2}\delta(y):=\frac{1}{2}\text{dist}(y,\partial K)$ we have 
    \begin{equation}\label{lower}
        \G(y,z)\geq T(M,\lambda)\frac{1}{|z-y|^{n-2}}.
    \end{equation}
\end{theorem} 

Here, the pole of Green's function has been considered to be only at the origin. While we emphasize that this manuscript and \cite{Pat25} deal with Dirichlet Green's functions, we will simply call it the Green's function in this manuscript, without any scope for confusion.

We can use \cref{thm2} and a standard argument using Harnack inequality and repeated change of poles of Green's function to show the following result.
\begin{theorem}\label{thm3}
    Consider the operator of \cref{Laplacian} and the Dirichlet Green's function corresponding to this operator. Further suppose that $\B\in C^{0,\alpha}(\Omega')$ for some $\alpha>0$, for any compactly supported subdomain $\Omega'\subset K$, and that $\B$ satisfies  \cref{driftt,nonpositivee}.  Then for any $y\in K_r:=\{x:\delta(x)\geq r\}$ with $r<1$, and for any $x\in \Omega$ with $|x-y|\leq \frac{1}{2}\delta(y)$, we have for some constant $T_{1}^{'}$ dependent on $M,r$ only,
    \begin{equation}
        \G(x,y)\leq T_{1}^{'}(M,r)\frac{1}{|x-y|^{n-2}}.
    \end{equation}
\end{theorem}

Since the drift is in $C^{0,\alpha}$ in compact subsets, one gets that the Green function is also locally in $C^{2,\alpha}$ in domains bounded away from both the boundary and the pole.


In \cite{Naz12} Nazarov and Ural’tseva study the properties of the solutions for elliptic operators with  drifts whose divergence is negative in distributional sense, in the scale of Morrey spaces.

Gruter and Widman studied the Green function for the divergence form operator in \cite{GrWi}. In \cite{MM22} McOwen and Maziya study the fundamental solution for non-divergence principal terms and drifts satisfying point divergences with a square Dini condition on the coefficient of the drift. Previous results of Cranston-Zhao, Ifra-Riahi, Mourgoglou, Kim-Sakellaris, Sakellaris \cite{An86, CrZh87,IfLo05,Singdr,KimSa,Morg,Sake} have studied drifts in suitable $L^p$ spaces, certain Kato and Lorenz spaces. In particular, essentially the necessity of $|\B|^2\in K_{n}$, a suitable Kato space, has appeared in the arguments of \cite{CrZh87, IfLo05,Morg}. It is easy to see that the drift in consideration in this paper does not satisfy the above condition.\footnote{See the detailed discussion on the appropriate Kato spaces, in \cite{Morg}}

 In a separate direction, in work of Ancona and Hofmann-Lewis \cite{An86,Singdr}, one shows existence of solutions and pointwise estimates for Green functions for drifts that behave like $\epsilon/\delta(X)$ for some sufficiently small $\epsilon$. Such a drift also does not satisfy $|\B|^{2}\in K_n$.

The novelty of \cite{Pat25} and the present work lie in a regime where perturbative “hiding” estimates are not available for the lower order terms of the bilinear form corresponding to \cref{Laplacian}. In particular, we do not assume any arbitrarily small prefactor as above. 

An outline of the arguments of this paper is given in Section 4.3 .


\section{Notation and terminology.}

We call the diameter of $K$ as $\text{diam}(K)=D$. By a maximum point on a co-dimension one compact hypersurface, we mean a point, not necessarily unique, where the Green's function is maximized on that hypersurface. By a minimum point on a co-dimension one hypersurface, we mean a point, not necessarily unique, where the Green function is minimized on that hypersurface. These hypersurfaces under consideration will be certain spheres, and sets of the form $\partial K_y$ where $$K_y:= (1-y)K +yB_1$$ in a Minkowski sum defined during the proof of \cref{thm2}.

We term the `near field effect' as the argument of looking at the minimum points within the ball $B(0, 1/L)\setminus B(0, b_k)$, and controlling the gradients at the minimum points of $B(0,1/L)$. By the `far field effect' we mean the effect of looking at the control of the decay of the gradient in the region $K\setminus B_1$, as articulated in Subsection 2.1. We have used the term `far field' a bit loosely since in the far field case, we will be controlling the gradient at maximum points in the entire domain $K\setminus B(0,b_j)$ for sequences of points $b_j|_{j\geq 1}$ that are arbitrarily close to the origin. We will use this to control the maximum value of the Green's function on the the sphere $S(0,1/L)=\partial B(0,1,L)$. 

We also denote by $W^{1,2}_{0}(K)$ the closure of the space of smooth compactly supported functions on $K$, in the $W^{1,2}(K)$ norm. We denote the norm of a function $f\in L^{2}(K)$ as $||f||_2$, and the $W^{1,2}(K)$ norm of a function $g\in W^{1,2}(K)$ as $||g||_{1,2}:=||g||_2 +||\nabla g||_2$.
 

\section{Outline for existence of solutions to $L, L_{T}$.}

Below we fix a positive value of $\eps$ small enough so that 
\begin{align}\label{eps}
    1-(M^{\frac{1}{\beta}}+C_H)\eps -(C_H)^{2} \eps^2 >\frac{1}{2},
\end{align}

where $C_H$ is the constant arising from the Hardy inequality of \cref{hardy}.

We outline the proof of the following: 

 \begin{theorem}\label{existence}
     Consider the operators of \cref{imp22} with the drift term $|\B|$ satisfying the \cref{driftt}. Then there exists a weak solution $u\in W^{1,2}_{0}(K)$ that satisfies $L_1 u=g+\partial_i f_i$. With the assumption of \cref{driftt}, under the further hypothesis of \cref{nonpositivee} on the divergence of the drift, we also have the existence of $v\in W^{1,2}_0(K)$ that solves $L^{T}_{1}v=g+\partial_i f_i$. Here, $g,f_i\in L^{2}(K)$.
 \end{theorem}

We refer the reader to Chapter 8 of \cite{Gt}. We are required to show the boundedness of the bilinear form as well as a modified argument in order to establish the coercivity. It is clear that for any $u, v \in W^{1,2}_0 (K)$, the boundedness of the bilinear forms corresponding to $L +\sigma I $ and $L_T  +\sigma I$ with bounds dependent on $\sigma$ is equivalent respectively to the boundedness of the bilinear forms corresponding to $L$ and $L_T$. This fact is used in the argument subsequent to Lemma 8.4 of \cite{Gt}.

The bilinear form corresponding to $L$ and $L_T$ are given respectively by 
\begin{equation}\label{forms}
    \mathcal{L}(u,v)=\int_{K} \partial_i u \partial_i v -\B_i (\partial_i u) v, \\
     \  \mathcal{L}_T(u,v)=\int_{K} \partial_i u \partial_i v + \B_i (\partial_i u) v -(\nabla\cdot \B)uv.
\end{equation}

First note that for any $\epsilon$ chosen sufficiently small, for all $x\in K$ with $\delta(x)\leq (\frac{\eps}{M})^{1/\beta}$, we have,
\begin{align*}
    \frac{M}{\delta(x)^{1-\beta}}\leq \frac{\eps}{\delta(x)}.
\end{align*}

For values of $\delta(x)> \delta_0 :=\big( \frac{\eps}{M} \big)^{1/\beta}$, the drift is bounded in magnitude crudely by 
\begin{align}
    \B_0 (\eps,M) := M\Big(\frac{1}{\delta_0}\Big)^{\frac{1}{\beta} -1}=\frac{M^{1/\beta}}{\eps^{1/\beta -1} },
\end{align}
in the interior of the domain $K_\eps:\{x\in K: \delta(x)> \big( \frac{\eps}{M} \big)^{1/\beta}\}$. Here, we have written $\B_0 (\eps/M)$ as a function of $M$, and we determine the small enough value of $\eps$ below.

In particular, for any point of the bounded domain $K$, we can add the two previous contributions, and consider that the drift is bounded crudely by 
\begin{align}
|\B|\leq \B_0 +\frac{\eps}{\delta(x)},
\end{align}
for a sufficiently small $\eps$ to be chosen later. Note also that the term $|\nabla\cdot \B|$ is also bounded in magnitude by $M/\delta(x)^{2-2\beta}$ by the condition of \cref{nonpositivee}. Thus we also have the condition that for any $x$ with $\delta(x)\leq \delta_0$, 
\begin{align}
    \frac{M^2}{\delta(x)^{2-2\beta}}\leq \frac{\eps^2}{\delta(x)^{2}}, 
\end{align}

and further that for $\delta(x)>\delta_0$ we can crudely bound the magnitude of the gradient $|\nabla\cdot \B|$ by 
\begin{align}
    \D_0 =\frac{M^{2}}{\delta_{0}^{2-2\beta}}
\end{align}

As before, we can consider that the total contribution to $|\nabla\cdot \B|$ is bounded by 
\begin{align}
|\nabla\cdot \B|\leq \D_0 +\eps^2/\delta(x)^{2}.
\end{align}

To see the boundedness of the two bilinear forms in \cref{forms}, we use the Hardy inequality in the domain $K$ (see for reference, \cite{Hof25, KLT11})
\begin{theorem}[Hardy inequality]\label{hardy}
    For any $u \in W^{1,2}_0(K)$, which is the closure in $W^{1,2}(K)$ of the space $C_{c}(K)$ of compactly supported functions in $K$, we have that \begin{align}
        \int_K \Big(\frac{u(x)}{\delta(x)}\Big)^{2}\leq C_H \int_K |\nabla u|^2
    \end{align}
\end{theorem}

Writing the upper bound on the drift and the upper bound on the gradient of the drift in terms of the sum of the two terms expressed earlier, and noting the fact that $u,v \in W^{1,2}_0 (K)$, using Holder inequality with exponents $p=q=2$, the Hardy inequality as above, and a crude bound in terms of $\B_0, \D_0$, we get the boundedness for both the forms.

\begin{theorem}
    We have, for any $u,v\in W^{1,2}_0 (K)$, positive numbers $K_1, K_2$ so that,
    \begin{align}
      |\mathcal{L}(u,v)|\leq K_1||u||_{1,2}||v||_{1,2}, \\|\mathcal{L}_T(u,v)|\leq K_2||u||_{1,2}||v||_{1,2}
    \end{align}
\end{theorem}

\begin{proof} We have,
    \begin{align}
|\mathcal{L}(u,v)| &= \left| \int_K (\partial_i u)(\partial_i v) - \mathcal{B}_i(\partial_i u)\, v \right| \notag \\
&\leq \int_K |\nabla u| \cdot |\nabla v| + |\mathcal{B}||\nabla u| v
\end{align}

where $|\nabla u| = \left(\sum_i (\partial_i u)^2\right)^{\frac{1}{2}}$, $|\nabla v| = \left(\sum_i (\partial_i v)^2\right)^{\frac{1}{2}}$, and we use the Cauchy--Schwarz inequality above.

Further, we have the crude bound $|\mathcal{B}| \leq \mathcal{B}_0(\eps,M) + \dfrac{\varepsilon}{\delta(x)}$,
for a value of $\varepsilon$ small enough to be determined later.

Recall that $\mathcal{B}_0(\eps,M) = \dfrac{M^{\frac{1}{\beta}}}{\varepsilon^{1/\beta}-1}$.

Thus, we write,
\begin{align}
\mathcal{L}(u,v) &\leq \left(\int_K |\nabla u|^2\right)^{\frac{1}{2}}
\left(\int_K |\nabla v|^2\right)^{\frac{1}{2}}
+ \int_K \left(\mathcal{B}_0(\eps,M) + \frac{\varepsilon}{\delta(x)}\right)|\nabla u|\, v
\end{align}

so,
\begin{align}
|\mathcal{L}(u,v)| &\leq \|\nabla u\|_2 \|\nabla v\|_2
+ \mathcal{B}_0(\eps,M) \|\nabla u\|_2 \|v\|_2 \notag \\
&\quad + \varepsilon \left(\int_K \left(\frac{v(x)}{\delta(x)}\right)^2\right)^{\frac{1}{2}}
\left(\int_K |\nabla u|^2 \right)^{\frac{1}{2}}
\end{align}

so,
\begin{align}
|\mathcal{L}(u,v)| &\leq \|\nabla u\|_2 \|\nabla v\|_2
+ \mathcal{B}_0(\varepsilon, M)\|\nabla u\|_2 \|v\|_2
+ C_{H}\varepsilon \|\nabla v\|_2 \|\nabla u\|_2
\end{align}
where in the last step we have used the Hardy inequality.

This proves the boundedness of the form $\mathcal{L}(u,v)$.

Now, consider the form $\mathcal{L}_T(u,v)$. We have,
\begin{align}\label{eq23}
|\mathcal{L}_T(u,v)| &= \left| \int_K \partial_i u\, \partial_i v
+ \mathcal{B}_i(\partial_i u)\,v - (\nabla \cdot \mathcal{B})\, uv \right|
\end{align}
The first two terms are dealt with as above, and it remains to
bound the last term involving $(\nabla \cdot \mathcal{B})$.

So, we have
\begin{align}
|\mathcal{L}_T(u,v)| &\leq (1+C_H \eps)\|\nabla v\|_2 \|\nabla u\|_2
+ \mathcal{B}(\varepsilon,M)\|\nabla u\|_2 \|v\|_2
+ \int_K |\nabla \cdot \mathcal{B}|\, |u|\, |v|
\end{align}

Again, we use the crude bound on $|\nabla \cdot \mathcal{B}|$, to get,
\begin{align}
|\mathcal{L}_T(u,v)| &\leq (1+C_H \eps)\|\nabla v\|_2 \|\nabla u\|_2
+ \mathcal{B}(\varepsilon,M)\|\nabla u\|_2 \|v\|_2 \notag \\
&\quad + \mathcal{D}_0 \int_K |u||v|
+ \varepsilon^2 \int_K \frac{|u||v|}{\delta(x)^2}
\end{align}

so,
\begin{align}
|\mathcal{L}_T(u,v)| &\leq (1+C_H \eps)\|\nabla v\|_2 \|\nabla u\|_2
+ \mathcal{D}_0 \|u\|_2 \|v\|_2 \notag +\mathcal{B}(\varepsilon,M)\|\nabla u\|_2 \|v\|_2 \\
&\quad + \varepsilon^2
\left(\int_K \left(\frac{u}{\delta(x)}\right)^2\right)^{\frac{1}{2}}
\left(\int_K \left(\frac{v}{\delta(x)}\right)^2\right)^{\frac{1}{2}}
\end{align}

Using Hardy inequality again, we get,
\begin{align}
|\mathcal{L}_T(u,v)| &\leq (1+C_H \eps) \|\nabla v\|_2 \|\nabla u\|_2
+ \mathcal{D}_0 \|u\|_2 \|v\|_2 \notag \\
&\quad + C_H \varepsilon^2 \|\nabla u\|_2 \|\nabla v\|_2
+ \mathcal{B}(\varepsilon,M)\|\nabla u\|_2 \|v\|_2
\end{align}
\end{proof}

In order to show the analoguos result for coercivity, it is enough to prove the version of Lemma 8.4 of \cite{Gt} for both the bilinear forms $\mathcal{L}, \mathcal{L}_T$.

\begin{theorem}\label{thm5}

We have, for a value of $\eps$ chosen from \cref{eps}, we have, 
\begin{align}   \mathcal{L}(u,u)\geq \frac{1}{2} \int_K |\nabla u|^2 -\Bigg( \frac{M^{\frac{1}{\beta}}}{\eps^{\frac{2}{\beta}-1}} \Bigg)  \int_{K} u^2 ,\\   \mathcal{L}_T(u,u)\geq \frac{1}{2} \int_K |\nabla u|^2 -\Big(  \frac{M^{\frac{1}{\beta}}}{\eps^{\frac{2}{\beta}-1}} + \frac{M^{\frac{2}{\beta}}}{{\eps^{\frac{2}{\beta}-2}}}  \Big)  \int_{K} u^2
\end{align}
\end{theorem}
\begin{proof}
We have,
\begin{align}
\mathcal{L}(u,u) &= \int_K \partial_i u \partial_i u + \mathcal{B}_i(\partial_i u)u
= \|\nabla u\|_2^2 + \int_K \mathcal{B}_i(\partial_i u)u.
\end{align}

Now as before, we have,
\begin{align}
\left|\int_K \mathcal{B}_i(\partial_i u)u\right|
&\leq \mathcal{B}_0(\eps,M) \|\nabla u\|_2 \|u\|_2 + C_H \varepsilon \|\nabla u\|_2^2 \notag \\
&= \frac{M^{1/\beta}}{\varepsilon^{1/\beta -1}} \|\nabla u\|_2 \|u\|_2 + C_H \varepsilon \|\nabla u\|_2^2
\end{align}

so,
\begin{align}
\left|\int_K \mathcal{B}_i(\partial_i u)u\right|
&\leq \frac{M^{1/\beta}}{2(\varepsilon^{1/\beta -1})}
\left(\varepsilon^{1/\beta}\|\nabla u\|_2^2 + \frac{1}{\varepsilon^{1/\beta}}\|u\|_2^2\right)
+ C_H \varepsilon\|\nabla u\|_2^2 \notag \\
&= M^{1/\beta}\frac{\varepsilon}{2} \|\nabla u\|_2^2 + C_H \varepsilon\|\nabla u\|_2^2
+ \frac{M^{1/\beta}}{2\varepsilon^{2/\beta -1}}\|u\|_2^2
\end{align}

Note that $M, \beta$ are the constants determining the drift, and the constant $C$ arises as a universal constant in Hardy's inequality.

Thus, we get from above that
\begin{align}
\mathcal{L}(u,u) \geq \|\nabla u\|_2^2 - M_1\varepsilon\|\nabla u\|_2^2
- \frac{M^{1/\beta}}{2\varepsilon^{2/\beta -1}}\|u\|_2^2,
\end{align}
where we have written $M_1 = (\frac{1}{2}M^{1/\beta} + C_H)$, and here it is enough to choose $\varepsilon$
small enough so that
\begin{align}\label{eqthree}
(1 - M_1\varepsilon) > \frac{1}{2}.
\end{align}

Finally, we deal with $\mathcal{L}_T(u,u) = \int_K \partial_i u\, \partial_i u
- \mathcal{B}_i(\partial_i u)u - (\nabla\cdot\mathcal{B})u^2$.

As above, we have
\begin{align}
\left|\int_K \mathcal{B}_i(\partial_i u)u\right|
\leq M_1\varepsilon\|\nabla u\|_2^2 + \frac{M^{1/\beta}}{2\varepsilon^{2/\beta -1}}\|u\|_2^2.
\end{align}

Also, we have,
\begin{align}
\left|\int_K (\nabla\cdot\mathcal{B})u^2\right|
&\leq \int_K |\nabla\cdot\mathcal{B}|\,|u^2|
\leq \mathcal{D}_0\|u\|_2^2 + C_H \varepsilon^2\|\nabla u\|_2^2,
\end{align}
from the calculation following \cref{eq23}. Thus we have,
\begin{align}
\mathcal{L}_T(u,u) \geq \|\nabla u\|_2^2 - M_1\varepsilon\|\nabla u\|_2^2
- C_H \varepsilon^2\|\nabla u\|_2^2 - \mathcal{D}_0\|u\|_2^2
- \frac{M^{1/\beta}}{2\varepsilon^{1/\beta -1}}\|u\|_2^2
\end{align}
In this calculation, it is enough to choose $\varepsilon$ small enough 
\begin{align}\label{eq377}
(1 - M_1\varepsilon - C_H \varepsilon^2) > \frac{1}{2},
\end{align} 
thus satisfying \cref{eps}. We then also automatically satisfy the requirement of \cref{eqthree}, and thus it is enough to require \cref{eq377} to hold.

Thus, we get,
\begin{align}
\mathcal{L}(u,u) &\geq \frac{1}{2}\|\nabla u\|_2^2 - \frac{M^{1/\beta}}{\varepsilon^{2/\beta -1}}\|u\|_2^2,
\end{align}
\begin{align}
\mathcal{L}_T(u,u) &\geq \frac{1}{2}\|\nabla u\|_2^2 - \frac{M^2}{\delta_0^{2-2\beta}}\|u\|_2^2 - \frac{M^{1/\beta}}{2\varepsilon^{1/\beta -1}}\|u\|_2^2
\end{align}

Recalling that $\delta_0 = \left(\dfrac{\varepsilon}{M}\right)^{1/\beta}$, we get
\begin{align}
\mathcal{L}_T(u,u) &\geq \frac{1}{2}\|\nabla u\|_2^2
- M^2 \left(\frac{M}{\varepsilon}\right)^{(2-2\beta)/\beta} \|u\|_2^2
- \frac{M^{1/\beta}}{\varepsilon^{2/\beta-1}}\|u\|_2^2 \notag \\
&= \frac{1}{2}\|\nabla u\|_2^2
- \left(\frac{M^{2/\beta}}{\varepsilon^{2/\beta-2}} + \frac{M^{1/\beta}}{2\varepsilon^{2/\beta-1}}\right)\|u\|_2^2
\end{align}

\end{proof}

After this result is established the rest of the argument, using the Fredholm alternative, is identical to the argument presented after Theorem 8.4 of \cite{Gt}.
\section{Preliminaries}
 Note first that the maximum principle is applicable in this setting as in \cite{Pat25}. Recall that when we define, $f(s)=|\{x\in \Omega: G(x,0)\geq s\}|^{\frac{n-2}{n}}$, and we have the following two cases in \cite{Pat25}; 
 \begin{itemize}
    \item[(i)] There exists a large enough absolute constant $C$ so that for any configuration of the drift in the ball $B(0,1)$, we have $f(2s)\geq \frac{1}{C}f(s)$ for all $ s\geq \widetilde{s}:=\max\limits_{x\in S(0,\frac{1}{2})}\G(x,0) $.
    \item[(ii)] The above fails, and so for all positive integers $N$ large enough, we have some drift configuration $|\B_N|$ for which we have $f(s_N)>N f(2s_N)$ for some $s_N\geq \widetilde{s_N}$.  Here we define, $\widetilde{s_N}:=\max\limits_{x\in S(0,\frac{1}{2})}\G_N(x,0) $, where $\G_N(x,0)$ is the Green's function corresponding to the operator with the drift $|\B_N|$. We deal with this case later, in a manner similar to the argument used in the later part of the argument for Case $(i)$.
\end{itemize}
\bigskip

In Case 1, using an interpolation argument, we are reduced to showing that one cannot have a sequence of drifts $B_i$ for which one has for the corresponding Green's function,
\begin{align}\frac{|\{x : G_i(x, 0) \ge 2\alpha_i^*\}|}{|B(0, \frac{1}{L})|} \to 0, \text{ as } i \to \infty.\end{align}

We call $y_{i,\text{min}}$ the minimum distance $\delta_{\text{min},i}$ from the origin to the level set where $G_i (x,0)=2\alpha^{*}_i$.

Here $\alpha^{*}_i$ is the maximum value of the Green function $G_i(x,0)$ corresponding to the drift $\B_i$, on the level set $B(0,1/L)$. Here $L$ is chosen large enough so that
\begin{align}\label{eq10}
B(0,1/L)\subset B_1:=B(0,\frac{1}{2}R).
\end{align}

\subsection{Points of tangency of level sets with $\partial K_y$}
Consider the $i$ value large enough, and the point $y_{i, \min}$ on the boundary $\partial \Omega_{2\alpha_i^*}\subset B(0,1/L)$, that is a minimum distance $\delta_{\min, i}$ from the origin. As a consequence of the above, for the fixed $L$, we have $\delta_{\min, i} \to 0$ as $i \to \infty$. We have $B(0, \delta_{\min, i}) \subset \Omega_{2\alpha_i^*}$. By Harnack inequality, there is a constant $\beta_1 >1$, independent of $r$ so that for all points on the sphere $S(0,r)$, the values of the Green function are comparable;
\begin{equation}
    \text{sup}_{S(0,r)}G_i(\cdot)\leq \beta_1 \ \text{inf}_{S(0,r)}G_i(\cdot).
\end{equation}

For any point on the boundary $\partial K_y$, for any $0 \le y \le 1$, for the maximum value of the Green's function on $\partial K_y$, attained at $z_y \in \partial K_y$, one uses the fact that $\partial K_y$ is convex with respect to the outward unit normal, and thus if one chooses local Cartesian coordinates with one coordinate $x_1$ along the outward unit normal to $\partial K_y$ at $z_y$, and other coordinates $(x_2, x_3, \dots, x_n)$, we see that along any coordinate direction $x_i$ with $2 \le i \le n$,
\[
G(z_y + \hat{x}_i h) - G(z_y) = \frac{\partial G}{\partial x_i}(z_y)h + \left( \frac{\partial^2 G}{\partial x_i^2}(z_y + \hat{x}_i \theta h) \right) \frac{h^2}{2}
\]
Here $0 \le \theta \le 1$.

By definition, at this maximum point we have,
\begin{align}\frac{\partial G}{\partial x_i}(z_y) = 0, \ \ \text{for all} \ \ 2 \le i \le n,
\end{align}
 since the gradient points along the unit normal at the point of tangency.

Thus, we have $G(z_y + \hat{x}_i h) - G(z_y) = \frac{\partial^2 G}{\partial x_i^2}(z_y) \frac{h^2}{2} + O(h^{2+\alpha})$.

By definition we have $G(z_y + \hat{x}_i h) - G(z_y) \le 0$, so if $\frac{\partial^2 G}{\partial x_i^2}(z_y + \hat{x}_i h) > 0$ then for $h$ small enough depending on the implicit parameters, we get that $$\frac{\partial^2 G}{\partial x_i^2}(z_y + \hat{x}_i h) \frac{h^2}{2} + O(h^{2+\alpha}) > 0,$$ which is a contradiction. Thus we have that for each $2 \le i \le n$, $\frac{\partial^2 G}{\partial x_i^2}(z_y) \le 0$.

We have the pointwise condition that, $$\sum_{i=1}^{n} \frac{\partial^2 G}{\partial x_i^2}(z_y) + \vec{B}(z_y) \cdot \frac{\partial G}{\partial x_i}(z_y) = 0.$$

Thus we must have at each such maximum point
\[
\frac{\partial^2 G}{\partial x_1^2}(z_y) + B(z_y) \frac{\partial G}{\partial x_1}(z_y) \ge 0.
\]

Now consider any maximum point on the sphere $S(0, r)$, $r < \frac{R_j}{2}$. In that case, we have, by an argument similar to the above, that in polar coordinates in $\mathbb{R}^n$ with $n \ge 3$,
\begin{equation}
0 \ge G(z_y + \hat{\theta}_i h_{\theta_i}) - G(z_y) = \frac{\partial G}{\partial \theta_i}(z_y)h_{\theta_i} + \frac{\partial^2 G}{\partial \theta_i^2}(z_y + \tilde{\theta}_i h_{\theta_i}) \frac{h_{\theta_i}^2}{2} \quad 
\end{equation}
Here $\tilde{\theta} \le 1$. Again we have that the gradient points along the unit normal vector, and so $\frac{\partial G}{\partial \theta_i} = 0$. Further from homogeneity alone, one sees that,

\begin{equation}\frac{\partial^2 G}{\partial \theta_i^2} = \sum_{m,n} f_{m,n}^{(i)}(\vec{x}) \frac{\partial^2 G}{\partial x_m \partial x_n} + \sum_{m} g_m^{(i)}(\vec{x}) \frac{\partial G}{\partial x_m},\end{equation} where $f_{m,n}^{(i)}$ are homogeneous and quadratic, whereas $g_m^{(i)}$ is linear in the $x_i$ coordinates. Thus it is easy to see that the $C^\alpha$ regularity of the double derivatives of $G$ in terms of Cartesian coordinates, means the $\frac{\partial^2 G}{\partial \theta_i^2}$ are also $C^\alpha$ in terms of the angular coordinates, with implied constants depending on the domain $K_Kappa = K_{y_j} \setminus B(0, b_{Kappa_0})$.

Thus we have that
$$0 \ge \frac{\partial^2 G}{\partial \theta_i^2}(z_y) \frac{h_{\theta_i}^2}{2} + O(h_{\theta_i}^3).$$

Note that $h_{\theta_i}$ is an angular variable. Thus again we must have $\frac{\partial^2 G}{\partial \theta_i^2} \le 0$ for all $1 \le i \le n-1$, otherwise we have a contradiction as before. From the structure of the Laplacian in spherical coordinates in dimensions $n \ge 3$ (see for example, Section 2 of \cite{Sm19}, in particular Eq 2.12), this implies that
$$\frac{\partial^2 G}{\partial r^2} + \frac{n-1}{r} \frac{\partial G}{\partial r} + B(r) \frac{\partial G}{\partial r} \ge 0 \quad \text{at } z_y \in \partial K_y \text{ (where } \max_{S(0,r)} G(\cdot) = G(z_y))$$

By essentially an identical argument as before, we get that for any minimum point $z_y$ on the sphere, we have
$$\frac{\partial^2 G}{\partial r^2}(z_y) + \frac{n-1}{r} \frac{\partial G}{\partial r}(z_y) + B(r) \frac{\partial G}{\partial r}(z_y) \le 0.$$
\subsection{Sequences of points converging to pole with control on gradients.}
We recall the following result from \cite{Pat25}.

\begin{theorem}\label{lemma4} There is a uniform constant $M_0$ so that for any configuration of the drift that satisfies the ambient conditions, we have a sequence of maximum points, and a sequence of minimum points $b_k|_{k=1}^\infty$ so that
\begin{equation}
\left| \frac{\partial G}{\partial x} \right|_{b_k} b_k^{n-1} < M_0 \left| \frac{\partial G}{\partial r} \right|_{a_k} a_k^{n-1}, \quad \text{and } a_k \to 0, b_k \to 0
\end{equation}
as $k \to \infty$.
\end{theorem}

In the above equation, it is understood that the pole of the Green's function is at 0.

We need a variant of this, which we establish below with the help of \cref{thm1}.

\begin{theorem}\label{lemma5} There is a uniform constant $C_0$, so that for any configuration of the drift that satisfies the ambient conditions, we have a sequence $\bar{a}_k|_{k=1}^\infty$ (possibly distinct from the sequence of minimum points of \cref{lemma4}) of minimum points for which we have,
\begin{equation}
\left| \frac{\partial G}{\partial r} \right|_{\bar{a}_k} \bar{a}_k^{n-1} \ge C_0.
\end{equation}
\end{theorem}

\begin{proof}[Proof of \cref{lemma5}] Assume to the contrary that for any arbitrary small $C_0$, there exists some $i(M)$ and drift $B_{il}$ satisfying
the ambient conditions, so that eventually for any sequence of minimum points $\{a_{i,k}\}_{k=1}^{\infty}$ with $a_{i,k} \to 0$ as $k \to \infty$, so that
$$S := \sup \left( \left| \frac{\partial G}{\partial x} \right|_{a_{i,k}} a_{i,k}^{n-1} \right) < C$$
Thus, for each $k$, we have $\left| \frac{\partial G}{\partial x} \right|_{a_{i,k}} < \frac{S}{a_{i,k}^{n-1}}$.

We start with an arbitrary minimum point $a_{i,k_0}$. Initially we choose the increment $h_0$ so that
$$h_0 \cdot \sup_{B(0, a_{i,k_0}) \setminus B(0, \frac{1}{2} a_{i,k_0})} |G_i''(x)| \ll \frac{S}{a_{i,k_0}^{n-1}}, \quad h_0 \ll \frac{1}{2} a_{i,k_0}.$$

Now we define the sequence the following way: first consider the point $a_{i,(k_0+1)}$ which is the minimum point of the level set $\{x : G(x) = G(a_{i,k_0}) - h_0\}$. By definition we have $|a_{i,k_0} - h_0| \ge |a_{i,k_0+1}|$.

We continue this process till we reach some $t_1 > 1$ so that $a_{i, k_0+t_1-1} \ge \frac{1}{2} a_{i,k_0}$ and $h_0' = |a_{i, k_0+t_1-1} - \frac{1}{2} a_{i,k_0}| \le h_0$, while at the last stage we define $a_{i, k_0+t_1}$ as the minimum point of the level set $G(a_{i, k_0+t_1-1} - h_0') = G(\frac{1}{2} a_{i, k_0+t_1})$.

Now similarly we define the sequence $\{a_{i, k_0+t_m}\}_{m=1}^{\infty}$ for each $m > 1$ in the same way as above. In other words, we have for each $a_{i, k_0+t_m+1}$ is the minimum point of the level set $\{x : G(x) = G(a_{i, k_0+t_m} - h_m)\}$, where we have defined $h_m$ as
$$h_m \cdot \sup_{B(0, a_{i,k_0}) \setminus B(0, \frac{1}{2} a_{i, k_0+t_m})} |G_i''(x)| \ll \frac{S}{a_{i,k_0}^{n-1}}, \quad h_m \ll \frac{1}{2} a_{i, k_0+t_m}.$$

Then we further define the points $a_{i, k_0+t_m+j}$ for each $j \ge 1$ till we reach the point $a_{i, k_0+t_m+j-1} \ge \frac{1}{2} a_{i, k_0+t_m}$ with $h_m' = |a_{i, k_0+t_m+j-1} - \frac{1}{2} a_{i, k_0+t_m}|$ and $h_m' \le h_m$, and we define the point $a_{i, k_0+t_{m+1}}$ as the minimum point of the level set $G(a_{i, k_0+t_{m+1}} - h_m') = G(\frac{1}{2} a_{i, k_0+t_m})$. At each step the decrement is $h_m$ with $a_{i, k_0+t_m+j}$, for $1 \le j \le t_{m+1}-t_m$, being the minimum point of the level set of $G(a_{i, k_0 + t_m + j - 1}- h_m)$.

Since for each $m \ge 1$, we have a sequence of points $a_{i, k_0 + t_m + j}$, with $1 \le j \le t_{m+1} - t_m$, so that $a_{i, k_0 + t_{m+1}} \le \frac{1}{2} a_{i, k_0 + t_m}$ this sequence clearly converges to $0$ as $m \to \infty$. We continue this iteration till we reach some $t_{m_0} = N$, so that $a_{i, k_0 + N} = \varepsilon$ so that $G(\varepsilon) > K(M, \lambda) (\frac{1}{\varepsilon})^{n-2} \gg G(a_{i, k_0})$, where we have invoked the result of \cref{thm1}.

Now, we have at each stage,
\begin{equation} \label{eq25}
G(a_{i, k_0 + t_m + j - 1} - h_m) - G(a_{i, k_0 + t_m + j - 1}) = \left| \frac{\partial G}{\partial r} \right|_{a_{i, k_0 + t_m + j - 1}} \cdot h_m + O_m(h_m^2)
\end{equation}
Thus, we have,
\begin{align*}
\sum_{q=0}^{N} \frac{h_q S}{(a_{i, k_0 + q})^{n-1}} + O_q(h_q^2) &\ge G(a_{i, k_0+N}) - G(a_{i, k_0}) \approx G(\varepsilon) \\
&\gg G(a_{i, k_0})
\end{align*}
where we have defined $t_{m_0} = N$, and taken $h_q = h_m$ whenever $k_0 + t_m \le q \le k_0 + t_{m+1}$.

From the above, we thus get with an altered constant $S'$, that,
\begin{equation} \label{eq26}
\sum_{q=0}^{N} \frac{h_q S'}{(a_{i, k_0 + q})^{n-1}} \ge G(\varepsilon)
\end{equation}
Clearly, we also have, by \cref{thm1}, that as mentioned above, that $G(\varepsilon) > K(M, \lambda) (\frac{1}{\varepsilon})^{n-2}$.

But by a basic integral bound, and the construction of the minimum points $\{a_{i, k_0 + t_m + j}\}_{j=1}^{t_{m+1}-t_m}$, from (\ref{eq26}), we see that 
\begin{align*}
G(\varepsilon) &\ll K(M, \lambda) \left( \left(\frac{1}{\varepsilon}\right)^{n-2} - \left(\frac{1}{a_{k_0}}\right)^{n-2} \right) \\
&\le K(M, \lambda) \left(\frac{1}{\varepsilon}\right)^{n-2}
\end{align*}
This is a contradiction, which proves the lemma.\end{proof}
\subsection{Outline of subsequent argument}
We outline how the argument of this article is modified, in comparison to the original argument for the unit ball. In \cite{Pat25}, we consider the points of maximum and minimum for particular level sets. In the "near field" case, this involved moving "outwards" from the minimum point $z_y$ of a given level set, by a given infinitesimal amount. We then consider the value of the Green's function $G(z_y + h)$, and consider the minimum point $z_{y+dy}$ of the level set $\{y : G(y) = G(z_y + h)\}$ which by definition is at a distance at most $h$ from the level set $\{y : G(y) = G(z_y)\}$. Then we invoke the differential inequalities at each of the points $z_y$ and $z_{y+dy}$.

In the far field case, we looked at point $z_y$ of maximum of a given level set, and moved infinitesimally ``inward" to a new level set and considered the maximum point $z_{y-dy}$ of the new level set $\{y : G(y) = G(z_y - h)\}$.

In either of these cases, we have to deal with situations where radial values of these maximum or minimum points converge, which is then ruled out by invoking a-priori bounds on the Green's function  in domains of the form $A_{ij}$.

More crucially, the geometry of the unit ball is necessary in the argument of \cite{Pat25} in the far field situation, to ensure that the bounds from the exponential factor coming from the drift is bounded. In case the geometry of the domain is not the unit ball, it might happen in the far field case that while we move infinitesimally inward towards the origin in successive infinitesimal steps, we remain arbitrarily close to the boundary for long times.

In this article, we consider one modification of the earlier method, that is able to extend the earlier result to convex domains, using a Minkowski interpolation between the original domain $K$ and an interior ball $B_1$. By a careful modification of the existing argument of \cite{Pat25}, one then considers each such domain $K_y = (1-y)K_0 + y B_1$. One main departure from the previous argument, is that within $K\setminus B_1$, for one differential element $t_y$, the subsequent differential element might be significantly larger than the $t_y$, but remains controlled depending on the maximal ratio of the support functions across different directions for $K$. We are able to show that the gradient of the Green function, once with enough decay in $K\setminus B_1$, cannot then increase arbitrarily, which is sufficient for our purpose. 

In the far field case within $K\setminus B_1$, one moves inward from $z_y$ by some amount $t_y$. Instead of looking at the level set of the Green's function $\{y : G(y) = G(z_y - t_y)\}$, one instead looks at the maximum value of $G$ on the unique boundary $\partial K_{y+dy}$ on which $(z_y - t_y)$ lies. A careful modification of the original argument shows this also gives the expected decay of the Green's function using the differential inequality. Moreover, one does not have to keep track of possible accumulations of the maximum points, since now the maximum points under consideration lie on successive $\partial K_{y+ndy}$, for integer $n \ge 0$.

More importantly, for the convex domain $K$ considered here, this means these maximum points are isolated at controlled distances away from the boundary in a way that ensures that the radial integral of the drift converges. As a result, in the domain $K \setminus B_1$, one needs to consider the supporting hyperplanes of infinitesimally separated domains $K_y$ and $K_{y+dy}$, and modify the differential argument to this setting.

It is also easy to see that the integral of the drift is bounded when one considers the far field effect; it is easily seen to be bounded by the integral of the upper bound of the form $\frac{M}{\delta(x)^{1-\beta}}$ along the direction that has the $\hat{n}$ so that $h_K(\hat{n}) = \min_{S^{n-1}} h_K(y)$.

Thus, as in Equation (59) of \cite{Pat25}, in Case (1), we need to show that one cannot have 
$$\frac{|\{x : G_i(x, 0) \ge 2\alpha_i^*\}|}{|B(0, \frac{1}{L})|} \not\to 0, \text{ as } i \to \infty.$$

For further progress on this and related questions, it is plausible that both the original method of \cite{Pat25} as well as the modifications introduced in this manuscript, might prove to be useful in dealing with this question with more general non convex domains, and with poles of the Green function arbitrarily close to the boundary. One would plausibly study these methods in many other problems in partial differential equations involving potential terms and also the landscape function as considered in \cite{DGM23, Pog24} in the setting of similarly singular potentials.

\begin{remark}
    We remark that one of the motivations for studying the question in domains more general than the unit ball, is to consider a Bilipshitz mapping to transfer the problem to perturbed domains, including to the unit ball, where the drift is altered accordingly, while still diverging near the boundary in the manner of \cref{Laplacian}, and where the principal term gets altered to a divergence form operator. 
\end{remark}

\begin{remark}
    It follows from the structure of the Harnack constant, see for example Theorem 8.20 of \cite{Gt}, that the Harnack inequality is applicable in the setting of the drift in \cref{Laplacian} in a bounded domain. More generally, the Harnack inequality for such operators with drifts in more general Morrey spaces, was established in \cite{Naz12}.
\end{remark}

\begin{remark}
    We have stated the bounds as above for simplicity, but in general one should be able to consider $|\B|^2$ and $|\nabla\cdot \B|$ to belong respectively to the more general Morrey spaces $M^{\frac{n}{2-2\beta}}$ and $M^{\frac{n}{2-\beta}}$ as considered in \cite{Ha24}. In that case, the proof for the lower bound on the Green's function, which has been worked out in particular for the case of the drift of the form $M/\delta(X)^{1-\beta}$ can also be appropriately modified.  We also note, as mentioned in the introduction, that drifts that are bounded like $\eps/\delta(X)$ have been considered in several places in the literature, such as \cite{An86, Singdr}, for sufficiently small $\eps$, for which Green's functions are bounded pointwise for poles arbitrarily close to the boundary. For the case of the unit ball, there is an example constructed with a radial drift with $\eps=1$, for which solutions do not exist \cite{Pat24}. In future work, we will study the applicability of the methods in this paper, for a wider class of function spaces.

\end{remark}
\begin{remark}
    No assumptions have been made in this paper, on the direction of the drift near the boundary, and the magnitude of the drift is enough to satisfy the pointwise estimate on the Green functions. 
\end{remark}
  In this paper \cref{lemma5} is new and provides an alternate streamlined proof in the far field case, however a  version of Lemma 9 and 10 of \cite{Pat25}  modified from the setting of the unit ball to this domain $K$, also completes the argument of this paper.

\section{Proof of \cref{thm2} for Case 1.}
The setup is as follows: we assume that as $i\to \infty$ we have $y_{i,\text{min}}\to 0$. For each such $i$, we choose a pair $a_{i,k_0},b_{i,k_0}$ of minimum and maximum points and a minimum point $\bar{a}_{i,k_0}$ so that each of these has magnitude strictly less than $|y_{i,\text{min}}|$ and without loss of generality we take $|\bar{a}_{i,k_0}|\leq |a_{i,k_0}|$. We consider the sequence of domains $A_{i,j}:= K_{y_j}\setminus B(0,b_{i,k_0})$, taking $y_j \to 0$ as $j\to \infty$, and in each of these domains $A_{i,j}$, use a-priori bounds on the Green function and it's first and second derivatives and use the differential argument in this domain. As $y_j\to 0$, due to the continuity of the Green function on the boundary $\partial K$, the sequence of Green functions at the maximum points of $\partial K_{y_j}$ goes to $0$. This will give us the desired contradiction in Section 5.4 and Section 6.

We first deal with Case 1 above. We first deal with the far field effect, where the gradient of the Green's function is shown to decay at a controlled rate in $K\setminus B(0,b_k)$. In section 4.1 we show that once a sufficient amount of decay of the Green function has taken place at these infinitesimally separated maximum points in the region $K\setminus B_1$, the Green function does not again grow to arbitrarily large values.

\subsection{Control of decay of gradient of Green's function in $K\setminus B_1$.}

Recall that we are considering the $C^{1}$ convex body $K$, with the property that there is a center point $x_0$, and a ball $B_0 := B(0,R) \subset \operatorname{int}(K)$, the interior of $K$. We consider the Minkowski interpolation between the two convex bodies $K$ and $B_1 := B(0, \tfrac{1}{2}R)$, thus getting a family of convex bodies
\begin{equation}
 K_y := (1-y)K + y B_1 : 0 \le y \le 1 .
\end{equation}


We have the following standard fact for the support function through Minkowski addition and dilation, for each $0 \le y \le 1$:

\begin{equation}\label{1}
h_{K_y} = h_{(1-y)K + y B_1} = (1-y) h_K + y h_{B_1}.
\end{equation}

It follows that for any $y_1<y_2$, we have $K_{y_1}\supset K_{y_2}$.

It is enough here that $K$ is merely convex. As we consider the fixed body $K_{y_j}$ with a fixed $y_j \neq 0$, for any arbitrarily large $j$ we use dilation and Theorem 2 of \cite{KrPa91}: since $B_1$ has a $C_1$ boundary, the body $K_y$ for $y\neq 0$ also has a $C^1$ boundary. This is enough to determine unique normals and unique supporting hyperplanes everywhere on the boundaries of $K_y$ for any $1\geq y\geq y_j$, which is sufficient for the subsequent argument. Finally we take the limit as $j\to \infty$.

For each $0 \le y \le 1$, consider a point $s_y$ (not necessarily unique) where the Green's function $G(\cdot, x_0)$ is maximized on $\partial K_y$, the boundary of $K_y$. 

\begin{enumerate}
\item First assume, without loss of generality, that the first time $|\nabla G(\cdot, x_0)|\leq 1$ at any of the maximum points of $K_{y_j}$, is at some point $s_{y_*}$ on $\partial K_{y_*}\subset K\setminus B_1$. This means in particular that for all the maximum points $\{ s_v : v \le \tfrac{1}{2}R \}$, we have $|\nabla G(s_v, x_0)| > 1$.

\item The case where, the first time $|\nabla G(\cdot, x_0)| \le 1$ on any of the maximum points of $K_{y_{j}}$,
is on the boundary point of $\partial B(x_0, v_*)$ for some $v_* \le \tfrac{1}{2}R$.

\item The case where, $|\nabla G(\cdot, x_0)| \le 1$ throughout the set $\{ s_y : 0 \le y \le 1 \} \cup \{ s_v : b_{i,k}\le v \le \tfrac{1}{2}R \}$.

\end{enumerate}

\bigskip

Assume now that we are in case (1). In this case, we have that for $\{ s_y : y_* \le y \le 1 \}$, the maximum on $\partial K_y$ of $G(\cdot, x_0)$ is attained at points where $|\nabla G(\cdot, x_0)| > 1$. Recall that we called these points ``maximum'' points on $\partial K_y$.

Consider a decreasing sequence of values of $y_j \to 0$, with $y_j \ge 0$. In the domain $K_{y_j} \setminus B(x_0, b_{i,k})$, we get through the Schauder estimate that the Green function belongs to $C^{2,\alpha}$, thus with absolute bounds on the value of the Green's function and its first and second derivatives, with the bounds dependent on $K_{y_j} \setminus B(x_0, b_{i,k})$.

Starting with the domain $K_{y_j}$, we move ``inward'' towards $\partial B_1$ and show the growth of the Green's function. The argument is an extension of that in Lemma 8 of \cite{Pat25}. We write $G(y)$ to mean $G(s_y, x_0)$ with the fixed pole $x_0$.

We then have, starting with the point $s_{y_j}$, using the differential technique of Lemma 8 of \cite{Pat25}, the Minkowski interpolation of the bodies $K$ and $B_1$, the following argument.

First we note that while $y \le 1$, and $y > y_j$, we do not use $dy$ as the differential increment, since the actual differential increments depend on the differential increments of the support function of the bodies $K_y$. At this stage, it is the differential decrement denoted by $t_y$, which relates to $dy$ in the manner that will be described below.

We have
\begin{equation}
G(s_y - t_y) = G(s_y) - t_y G'(s_y) + \frac{t_y^2}{2} G''(s_y - \theta t_y),
\label{eq:1}
\end{equation}

Note that if $n_y$ is the unit outward normal to $K_y$ at $s_y$, we have $t_y$ as the decrement along the unique line through $s_y$, perpendicular to the support plane of $K_y$ passing through $s_y$, so that
\[
t_y = h_{K_y}(n_y) - h_{K_{y+dy}}(n_y).
\]

So,
\[
t_y = (1-y) h_K(n_y) + y h_{B_1}(n_y)
- \big( (1-(y+dy)) h_K(n_y) + (y+dy) h_{B_1}(n_y) \big).
\]

So,
\begin{equation}
t_y = -dy\big(h_{B_1}(n_y) - h_K(n_y)\big)=dy\Big(h_{K}(n_y)-\frac{R}{2}\Big).
\end{equation}

Note that the convex body $K$ is, by definition, such that the values
\[
dy\Big(h_{K}(n_y)-\frac{R}{2}\Big) 
\]
are comparable to each other, as the unit vector $ n$ ranges over the unit sphere $S^{n-1}$. Thus in the limit as $dy \to 0$, we have that
\[
t_y \to 0,
\]
whatever the value of $n_y$.

Note that the point $ s_{y+dy}$ lies on the surface $\partial K_{y+dy}$ by definition, and we have
\begin{equation}\label{eq200}
G(s_y, x_0) - G(s_y - t_y n_y, x_0)
< G(s_{y+dy}, x_0) - G(u_{y+dy}, x_0)
= G(s_{y+dy}, x_0) - G(s_y, x_0),
\end{equation}
where $u_{y+dy}$ is the point on the line
\[
\{ s_{y+dy} + \hat n_{y+dy} t : t \in \mathbb{R} \}
\]
so that
\[
G(u_{y+dy}, x_0) = G(s_y, x_0).
\]

Note that such a point $u_{y+dy}$ exists by the intermediate value theorem.

From \eqref{eq:1}, we have that, using the fact that $G \in C^{2,\alpha}$,
\begin{equation}
G(s_y - t_y) 
= G(s_y) - t_y G'(s_y) + \frac{t_y^2}{2} G''(s_y)
+ O(t_y^{2+\alpha}).
\label{eq:2}
\end{equation}

where the final term comes from the $C^\alpha$ estimate on the second derivatives of the Green's function. 

Thus we get
\begin{equation}
G(s_y) - G(s_y - t_y)
= t_y G'(s_y) - \frac{t_y^2}{2} G''(s_y)
+ O(t_y^{2+\alpha}).
\label{eq:3}
\end{equation}

So we get,
\begin{equation}
\frac{G(s_y) - G(\bar{s}_{y})}{t_y}
= G'(s_y) - \frac{t_y}{2} G''(s_y) + O(t_y^{1+\alpha}).
\label{eq:4}
\end{equation}

We also define the distance $w_y := |s_{y+dy} - u_{y+dy}|$ along the line
\[
\{ s_{y+dy} + \hat n_{y+dy} t : t \in \mathbb{R} \}.
\]
Thus by definition, we get that $w_y \leq |h_{K(n_{y+dy})}-\frac{R}{2}|$.

Thus we also get that
\begin{equation}
G(s_{y+dy} + w_y)
= G(u_{y+dy})
= G(s_{y+dy}) + w_y G'(s_{y+dy})
+ \frac{w_y^2}{2} G''(s_{y+dy})
+ O(w_y^{2+\alpha}).
\label{eq:5}
\end{equation}

Now, we just note, using an adaptation of the argument of the \cite{Pat25} paper, that
\[
\frac{\partial^2 G}{\partial x_i^2}(s_{y+dy}) \le 0
\quad \text{for all } i=1,\dots,n-1,
\]
where we have oriented the axes so that
\[
\frac{d^2}{dx_n^2} G(s_{y+dy}) = G''(s_{y+dy}),
\]
where this second derivative is taken along the line
\[
\{ s_{y+dy} + \hat n_{y+dy} t : t \in \mathbb{R} \}.
\]

Thus we must have
\[
G''(s_{y+dy}) + B \cdot \hat n_{y+dy} \, G'(s_{y+dy}) \ge 0,
\]
from the defining equation for the Green's function.

Thus, from \eqref{eq:5}, we get
\begin{equation}
G(u_{y+dy}) - G(s_{y+dy})
\ge w_y G'(s_{y+dy})
+ \frac{w_y^2}{2} \big(- B \cdot \hat n_{y+dy}\big) G'(s_{y+dy})
+ O(w_y^{2+\alpha}).
\label{eq:6}
\end{equation}

From here we get,
\begin{equation}
G(u_{y+dy}) - G(s_{y+dy})
\ge w_y G'(s_{y+dy})
\left( 1 - \frac{w_y}{2} B \cdot \hat n_{y+dy} \right)
+ O(w_y^{2+\alpha}).
\label{eq:7}
\end{equation}

Note that $G'(s_{y+dy}) \le 0$, and we get,


\begin{equation}
G'(s_{y+dy}) \le \left( \frac{G(u_{y+dy}) - G(s_{y+dy})}{w_y} + O(w_y^{1+\alpha}) \right) \frac{1}{\left( 1 - w_y \frac{1}{2} \vec{B} \cdot \hat{n}_{y+dy} \right)} \tag{5}
\end{equation}

Using the binomial approximation only up to the first order, we get,
\begin{equation*}
G'(s_{y+dy}) \le \left( 1 + w_y \frac{1}{2} \vec{B} \cdot \hat{n}_{y+dy} \right) \left( \frac{G(u_{y+dy}) - G(s_{y+dy})}{w_y} + O(w_y^{1+\alpha}) \right) \tag{5}
\end{equation*}

Further, we get, $G'(\bar{s}_y) = G'(s_y - t_y \hat{n}) = G'(s_y) - t_y G''(s_y) + o(t_y^{1+\alpha})$ again using the fact that $G$ is locally in $C^{2,\alpha}$ in this domain. 
We now use the fact that $0 \leq  G(\bar{s}_y) - G(s_y) \leq  G(s_{y+dy}) - G(u_{y+dy})$, and 
\begin{align*}
w_y &\le h_{K_y}(n_{y+dy}) - h_{K_{y+dy}}(n_{y+dy}) \\
&= \left( (1-y) h_K(n_{y+dy}) + y h_{B_1}(n_{y+dy}) \right) - \left( (1-(y+dy)) h_K(n_{y+dy}) + (y+dy) h_{B_1}(n_{y+dy}) \right).\end{align*}
So,
\begin{align*}w_y &\le dy (h_K(n_{y+dy}) - h_{B_1}(n_{y+dy})) = dy \left( h_K(n_{y+dy}) - \frac{R}{2} \right).
\end{align*}

Recall that $t_y = dy \left( h_K(n_y) - \frac{R}{2} \right)$. 
Consider the ratio $r_y := \frac{h_K(n_{y+dy}) - R/2}{h_K(n_y) - R/2}$, then we get from the above considerations, that
\begin{align*}
G'(s_{y+dy}) &\le \left( 1 + w_y \frac{1}{2} \vec{B} \cdot \hat{n}_{y+dy} \right) \left( \frac{G(s_y) - G(\bar{s}_y)}{dy (h_K(n_{y+dy}) - \frac{R}{2})} + O(w_y^{1+\alpha}) \right) \\
&= \left( 1 + w_y \frac{1}{2} \vec{B} \cdot \hat{n}_{y+dy} \right) \left( \frac{G(s_y) - G(\bar{s}_y)}{r_y t_y} + O(w_y^{1+\alpha}) \right) \\
&= \left( 1 + w_y \frac{1}{2} \vec{B} \cdot \hat{n}_{y+dy} \right) \left( \frac{1}{r_y} (G'(s_y) - \frac{t_y}{2} G''(s_y) + o(t_y^{1+\alpha})) + O(w_y^{1+\alpha}) \right)
\end{align*}

Again, we have $G''(s_y) + \vec{B} \cdot \hat{n}_y G'(s_y) \ge 0$, and so we get from above, that
\begin{equation*}
G'(s_{y+dy}) \le \left( 1 + w_y \frac{1}{2} \vec{B} \cdot \hat{n}_{y+dy} \right) \left( \frac{1}{r_y} (G'(s_y) + \frac{t_y}{2} \vec{B} \cdot \hat{n}_y G'(s_y) + o(t_y^{1+\alpha})) + O(w_y^{1+\alpha}) \right)
\end{equation*}

so,
\begin{equation}\label{eq3000}
G'(s_{y+dy}) \le \frac{1}{r_y} G'(s_y) \left( 1 + w_y \frac{1}{2} \vec{B} \cdot \hat{n}_{y+dy} \right) \left( 1 + t_y \frac{1}{2} \vec{B} \cdot \hat{n}_y \right) + O(t_y^{1+\alpha})
\end{equation}

We iterate this, in the domain $K_{y_j}$ for the fixed $j$, noting
that $t_y, w_y$ are both chosen small enough in comparison to
$\inf_{x \in K_{y_j}} d(x, \partial \Omega)$. In the process, we get in the
continued product, an expression of the form $\prod_{k=0}^{N} (r_{y+kdy})$ in the
denominator, which is bounded from above and below,
and in the limit as $dy \to 0$ and thus $N \to \infty$, we get

\begin{equation} \label{eq20}
G'(s_y) \le \overline{K} G'(s_{y^*}) \quad \forall  y > y^*, y\leq 1,
\end{equation}

once we show that the contribution of the drift term in the continued product is finite.
\begin{lemma}
    The contribution of the drift term in the continued product from \cref{eq20} is finite as $t_y\to 0$.
\end{lemma}

\begin{proof}
    It suffices to check that at each step of this infinitesimal iteration, we have an upper bound on the drift term that is integrable. The drift is bounded by $1/\delta(x)^{1-\beta}$ . For any point on $s_y \in K_y$, and with $n_y$ the unit normal to $K_y$ at the point $s_y$, suppose that $\delta(s_y):=|s_y - p_y|$, with $p_y \in \partial K$. Suppose that the unit normal to $K$ at $p_y$ is in the direction $m_y$ which in general does not coincide with $n_y$. In this case, we have,
\begin{equation}
    h_{K}(m_y)-h_{K_y}(m_y)= h_{K}(m_y)- (1-y)h_{K}(m_y)-yh_{B_1}(m_y)=y(h_{K}(m_y)-\frac{R}{2})
\end{equation}
This is the distance between the two parallel supporting hyperplanes each with unit normal $m_y$, that are tangent respectively to the bodies $K$ and $K_y$. 

Note that by construction, the point $s_y$ lies on the closed half space delineated by the supporting hyperplane for the body $K_y$ in the direction $m_y$, and the point $p_y$ lies on the opposite half space. Thus we have $\delta(s_y)\geq y(h_{K}(m_y)-\frac{R}{2}) \geq y(h_{min} -\frac{R}{2})$. Here $h_{min}$ is the minimum of the support function of $K$ over all directions.

Thus as $j\to \infty$, for each $0\leq y\leq 1$, we get an upper bound for the drift that is given by,
\begin{equation}
    \B(s_y)\leq \frac{M}{(y(h_{min}-\frac{R}{2}))^{1-\beta}},
\end{equation}
and which in the infintesimal iteration, gives a finite value.

\end{proof}

Thus from \cref{eq20} and the Hopf principle, we have
\begin{align}\label{eq3000}
|G'(s_y)| \ge \overline{K} |G'(s_{y^*})|, \forall y > y^*, y\leq 1.
\end{align}

Here the constant $\overline{K}$ only depends on the domain $K$, and the parameter
and in particular,
independent of the point $y^*$.

Note that 
\begin{equation}G'(s_{y+dy}) \le \frac{1}{r_y} G'(s_y) \left( 1 + w_y \frac{1}{2} \vec{B} \cdot \hat{n}_{y+dy} \right) \left( 1 + t_y \frac{1}{2} \vec{B} \cdot \hat{n}_y \right) + O(t_y^{1+\alpha})\end{equation}

so, \begin{equation}G'(s_{y+dy}) \le \frac{1}{r_y} G'(s_y) \left( e^{w_y \frac{1}{2} \vec{B} \cdot \hat{n}_{y+dy} + O\left(w_y^2 \left(\vec{B} \cdot \hat{n}_{y+dy}\right)^2\right)} \right) \left( e^{t_y \frac{1}{2} \vec{B} \cdot \hat{n}_y + O(t_y^2 \left(\vec{B} \cdot \hat{n}_y\right)^2}\right) + O(t_y^{1+\alpha}) \end{equation}

so, \begin{equation}\label{eq22}
G'(s_{y+dy}) \le \frac{1}{r_y} G'(s_y) e^{w_y \frac{1}{2} \vec{B} \cdot \hat{n}_{y+dy} + t_y \frac{1}{2} \vec{B} \cdot \hat{n}_y} + O\left(t_y^2 \left(\vec{B} \cdot \hat{n}_{y+dy}\right)^2\right) + O(t_y^{1+\alpha}),
\end{equation}

with altered constants in the $O$ term on the right, noting that on $K_{y_j}$,
$|G'|$ has a-priori bounds, due to the Schauder
estimates, which gets incorporated into the $O$ term constant.

Iterating equation \cref{eq22}, using the integrability of the drift
term, we arrive at \cref{eq20}


It remains to show that there is an absolute constant $W$, so 
that for any $0 \le y \le y^*$, we have
\begin{align}\label{eq3333}
|G'(s_y)| < W,  \ \ \text{for any}\ \ 0 \le y \le y^*.
\end{align}
Recall 
that we are in the case (1) where the point $s_{y^*}$ is on the boundary 
$\partial K_{y^*}$, where $K_{y^*} = (1-y^*) K + y^* B_1$ for some $y^* \le 1$.

This follows by applying the previous argument in the 
domain $K \setminus K_{y^*}$. Suppose to the contrary, that there 
is some $y_0 >0$ with $ y_0 \le y^*$ with $|G'(s_y)| \ge W$. In that case, the previous argument again gives us that,

\begin{equation} \label{eq333}
|G'(s_{y^*})| > \overline{K} |G'(s_{y_0})| \ge \overline{K} W\ge 1,
\end{equation}

where we have chosen 
$\overline{K}W > 1$. This contradicts the fact that $|G'(s_{y*})|<1$.

In cases (2) and (3) the analogous versions of \cref{eq3000,eq3333} are obtained with the same argument with minor modifications. These estimates for all the three cases, are used in Section 5.4 .

\begin{figure}[h]
\centering
\includegraphics[width=0.7\textwidth]{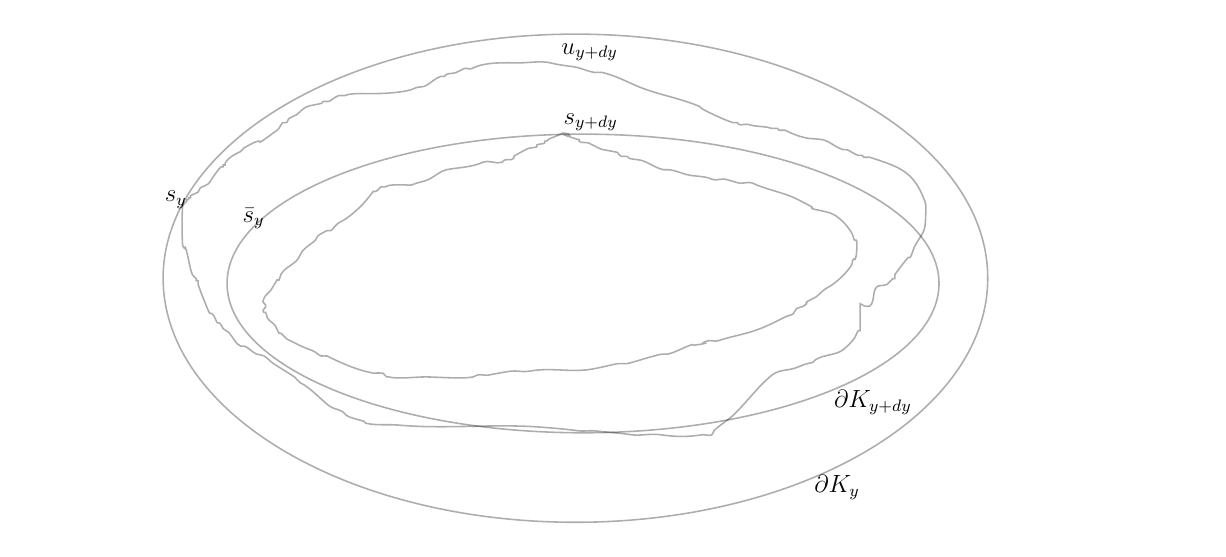}
\caption{The setting for the far field effect in $K\setminus B_1$, showing the maximum points on the surfaces $K_y$ and $K_{y+dy}$, which are the points where the Green function is maximized on the given sphere. We have $t_y=|s_y -\bar{s}_y|$, and $w_y=|s_{y+dy} -u_{y+dy}|$. Note that in general, $\bar{s}$ does not lie on $\partial K_{y+dy}$, and in general lies in $K_{y}\cap K^{c}_{y+dy}$.}
\label{far field}
\end{figure}

\subsection{Far field decay of the gradient of the Green's function within 
the annular region $B_1 \setminus B(0, \frac{1}{L})$}
Here we now estimate the decay of the gradient of the Green's function within 
the annular region $B_1 \setminus B(0, \frac{1}{L})$, as in \cite{Pat25} . We adopt the 
argument of \cite{Pat25} to this setting. Note that while in 
\cite{Pat25} we considered the maximum points on the level set 
from a given pole, 
here we consider the maximum values of the Green's function on 
the boundaries of convex sets whose boundaries that are Minkowski 
interpolations of $K$ and $B_1$. 

We thus have, as before, that,
\begin{equation}
G(s_y - t_y) = G(s_y) - t_y G'(s_y) + \frac{t_y^2}{2} G''(s_y) + O(t_y^{2+\alpha})
\end{equation}
so, $G(s_y) - G(s_y - t_y) = t_y G'(s_y) - \frac{t_y^2}{2} G''(s_y) + O(t_y^{2+\alpha})$.

Thus, exactly as before we get, defining $w_y = |u_{y+dy} - s_{y+dy}|$ along the line $\{ s_{y+dy} + \hat{n}_{y+dy} t : t \in \mathbb{R} \}$, we get,
\begin{equation} \label{9}
G(s_{y+dy} + w_y) - G(s_{y+dy}) = w_y G'(s_{y+dy}) + \frac{w_y^2}{2} G''(s_{y+dy}) + O(w_y^{2+\alpha})
\end{equation}

At this point, we use the polar form of Laplacian and get
\begin{equation}
G(s_{y+dy} + w_y) - G(s_{y+dy}) > w_y G'(s_{y+dy}) + \frac{w_y^2}{2} \left( - \frac{n-1}{s_{y+dy}} - \vec{B}(s_{y+dy}) \cdot \hat{n}_{y+dy} \right) G'(s_{y+dy}) + O(w_y^{2+\alpha})
\end{equation}

So, we have,
\begin{equation}
G(s_{y+dy} + w_y) - G(s_{y+dy}) \ge w_y G'(s_{y+dy}) \left( 1 - \frac{w_y(n-1)}{2 s_{y+dy}} - \frac{w_y}{2} \vec{B}(s_{y+dy}) \cdot \hat{n}_{y+dy} \right) + O(t_y^{2+\alpha})
\end{equation}

Noting again that $G'(s_{y+dy}) \le 0$, and that $w_y \le t_y$ (Note the difference from the analysis for the domain $K \setminus B_1$), we get that
\begin{equation}
G'(s_{y+dy}) \le \frac{1}{1 - w_y \left( \frac{n-1}{2 s_{y+dy}} + \frac{1}{2} \vec{B}(s_{y+dy}) \cdot \hat{n}_{y+dy} \right)} \left( \frac{G(s_{y+dy} + w_y) - G(s_{y+dy})}{w_y} + O(t^{1+\alpha}) \right)
\end{equation}

Using binomial theorem, we now get,
\begin{equation} \label{10}
G'(s_{y+dy}) \le \left( 1 + w_y \left( \frac{n-1}{2 s_{y+dy}} + \frac{1}{2} \vec{B}(s_{y+dy}) \cdot \hat{n}_{y+dy} \right) \right) \left( \frac{G(s_{y+dy} + w_y) - G(s_{y+dy})}{w_y} \right) + O(t^{1+\alpha})
\end{equation}

Using the fact that $0 < G(s_y) - G(\underline{s}_y) < G(s_{y+dy}) - G(u_{y+dy})$, and that $w_y \le t_y$, we get. (also recalling that $s_{y+dy} + w_y = u_{y+dy}$),

\begin{equation}
G'(s_{y+dy}) \le \left( 1 + w_y \left( \frac{n-1}{2s_{y+dy}} + \frac{1}{2} \vec{B}(s_{y+dy}) \cdot \hat{s}_{y+dy} \right) \right) \left( \frac{G(s_y) - G(\underline{s}_y)}{t_y} + O(t_y^{1+\alpha}) \right)
\end{equation}

Using and the earlier equation, we get,
\begin{equation}
G'(s_{y+dy}) \le \left( 1 + w_y \left( \frac{n-1}{2s_{y+dy}} + \frac{1}{2} \vec{B}(s_{y+dy}) \cdot \hat{s}_{y+dy} \right) \right) \left( G'(s_y) - \frac{t_y}{2} G''(s_y) + O(t_y^{1+\alpha}) \right)
\end{equation}

so, \begin{align}
G'(s_{y+dy}) \le \left( 1 + w_y \left( \frac{n-1}{2s_{y+dy}} + \frac{1}{2} \vec{B}(s_{y+dy}) \cdot \hat{s}_{y+dy} \right) \right) \Bigg( G'(s_y) - \\ \frac{t_y}{2} \left( \frac{n-1}{s_y} + \vec{B}(s_y) \cdot \hat{s}_y \right) G'(s_y) + O(t_y^{1+\alpha}) \Bigg)
\end{align}

so, \begin{multline} \label{eq42}
G'(s_{y+dy}) \le G'(s_y) + \frac{w_y}{2} \left( \frac{n-1}{s_{y+dy}} + \vec{B}(s_{y+dy}) \cdot \hat{s}_{y+dy} \right) G'(s_y) \\+ \frac{t_y}{2} \left( \frac{n-1}{s_y} + \vec{B}(s_y) \cdot \hat{s}_y \right) G'(s_y) + O(t_y^{1+\alpha})
\end{multline}

Analogous to the case of the proof of Lemma 8 of\cite{Pat25} we first use Taylor's theorem up to second order to conclude,
\begin{equation}
|G'(s_y)| t_y + \frac{t_y^2}{2} G''(s_y - \theta_1 t_y) \le |G'(s_{y+dy})| w_y + \frac{w_y^2}{2} G''(s_{y+dy} + \theta_2 w_y)
\end{equation}

so, \begin{align}|G'(s_{y+dy})| \ge \frac{1}{w_y} \left( |G'(s_y)| t_y + \frac{t_y^2}{2} G''(s_y - \theta_1 t_y) - \frac{w_y^2}{2} G''(s_{y+dy} + \theta_2 w_y) \right)\end{align}

so, \begin{equation} \label{12}
|G'(s_{y+dy})| \ge \frac{t_y}{w_y} |G'(s_y)| + \frac{t_y^2}{2w_y} G''(s_y - \theta_1 t_y) - \frac{w_y}{2} G''(s_{y+dy} + \theta_2 w_y)
\end{equation}

Consider the two separate cases:

\begin{itemize}
\item \begin{equation*}(a'): |G'(s_y)| w_y \le |G'(s_y)| t_y - \left( \left( \frac{n-1}{b_{k_0}} \right) |G'(s_y)| + 2 \sup_{K_{y_j} \setminus B(0, b_{k_0})} |G''(x)| \right) \frac{t_y^2}{2}\end{equation*}

so,

\begin{equation*}|G'(s_y)| \le \frac{t_y}{w_y} |G'(s_y)| - \frac{t_y^2}{2w_y} \left( \left( \frac{n-1}{b_{k_0}} \right) |G'(s_y)| + 2 \sup_{K_{y_j} \setminus B(0, b_{k_0})} |G''(x)| \right)
\end{equation*}

so, \begin{equation*}\frac{t_y}{w_y} |G'(s_y)| \ge |G'(s_y)| + \frac{t_y^2}{2w_y} \left( \left( \frac{n-1}{b_{k_0}} \right) |G'(s_y)| + 2 \sup_{K_{y_j} \setminus B(0, b_{k_0})} |G''(x)| \right)
\end{equation*}

So, using Equation \ref{12} above, we get, noting that $t_y \ge w_y$,
\begin{multline}
|G'(s_{y+dy})| \ge |G'(s_y)| + t_y \left( \frac{n-1}{b_{k_0}} \right) |G'(s_y)| + \frac{t_y^2}{w_y} \left( 2 \sup_{K_{y_j} \setminus B(0, b_{k_0})} |G''(x)| \right)\\ + \frac{t_y^2}{2w_y} G''(s_y - \theta_1 t_y) - \frac{w_y}{2} G''(s_{y+dy} + \theta_2 w_y)
\end{multline}

so, 
\begin{align}\label{eqa'}
|G'(s_{y+dy})| \ge |G'(s_y)| \left( 1 + t_y \left( \frac{n-1}{b_{k_0}} \right) \right)
\end{align}

\item (b') On the other hand, when we have the inequality opposite to (a') above, we get
\begin{equation}
\left( \left( \frac{n-1}{b_{k_0}} \right) |G'(s_y)| + 2 \sup_{K_{y_j} \setminus B(0, b_{k_0})} |G''(x)| \right) t_y^2 + |G'(s_y)| w_y > |G'(s_y)| t_y
\end{equation}

Noting that we have in this regime $|G'(s_y)| > 1$, we get
\begin{equation}
\left( \frac{n-1}{b_{k_0}} \right) t_y^2 + 2 \frac{\left( \sup_{K_{y_j} \setminus B(0, b_{k_0})} |G''(x)| \right)}{|G'(s_y)|} t_y^2 + w_y > t_y
\end{equation}

so, \begin{align}w_y \ge t_y - \left( \frac{n-1}{b_{k_0}} \right) t_y^2 - 2 \frac{\sup |G''(x)|}{|G'(s_y)|} t_y^2 \ge t_y - \left( \left( \frac{n-1}{b_{k_0}} \right) - 2 \sup |G''(x)| \right) t_y^2\end{align}

We get from \cref{eq42}, noting that $G'(s_y) \le 0$, when $$\left( \frac{n-1}{s_{y+dy}} + \vec{B}(s_{y+dy}) \cdot \hat{s}_{y+dy} \right) \ge 0,$$ that,
\begin{multline}
G'(s_{y+dy}) \le G'(s_y) + \frac{1}{2} \left( t_y - \left( \frac{n-1}{b_{k_0}} + 2 \sup_{K_{y_j} \setminus B(0, b_{k_0})} |G''(x)| \right) t_y^2 \right) \left( \frac{n-1}{s_{y+dy}} + \vec{B}(s_{y+dy}) \cdot \hat{s}_{y+dy} \right) G'(s_y) \\
+ \frac{t_y}{2} \left( \frac{n-1}{s_y} + \vec{B}(s_y) \cdot \hat{s}_y \right) G'(s_y) + O(t_y^{1+\alpha}).
\end{multline}

So, we get,
\begin{align}
G'(s_{y+dy}) \le G'(s_y) + \frac{t_y}{2} \left( \frac{n-1}{s_{y+dy}} + \vec{B}(s_{y+dy}) \cdot \hat{s}_{y+dy} \right) G'(s_y)\\ + \frac{t_y}{2} \left( \frac{n-1}{s_y} + \vec{B}(s_y) \cdot \hat{s}_y \right) G'(s_y) + O(t_y^{1+\alpha}) + O(t_y^2)
\end{align}

so, \begin{multline} \label{eqb'}
G'(s_{y+dy}) \le G'(s_y) + t_y \Bigg( \frac{n-1}{2} \Bigg( \frac{1}{s_y} + \frac{1}{s_{y+dy}} \Bigg)+\frac{1}{2} \Bigg( \vec{B}(s_{y+dy}) \cdot \hat{s}_{y+dy} + \vec{B}(s_y) \cdot \hat{s}_y \Bigg) \Bigg) G'(s_y) \\+ O(t_y^{1+\alpha})
\end{multline}

When $$\left( \frac{n-1}{s_{y+dy}} + \vec{B}(s_{y+dy}) \cdot \hat{s}_{y+dy} \right) < 0,$$ then we simply use the fact that $t_y \ge w_y$, to get the same.
\end{itemize}

Thus, again iterating inequalities the previous inequalities we get, and taking $dy \to 0$, $N \approx \frac{\text{diam}(K)}{dy}$, the contribution of the second order term goes to 0, and from the principal term we get in the case of (b') above, the solution to the differential inequality,
\begin{equation}
G'' + \left( \frac{n-1}{r} \right) G' + \vec{B} \cdot \hat{n} G' \ge 0
\end{equation}

Note that \cref{eqa',eqb'} are used over successive spheres that are infinitesimally separated by $t_y$, and the integration is direct and does not need to additionally rule out the possibility of the level sets accumulating arbitrarily close to each other in comparison to $t_y$ as in \cite{Pat25}.

\subsection{Near field case}

Consider now the minimum values of the Green function on the surfaces $\partial B_v$ with $a_{k_0} \le v \le \frac{1}{2} R$. Here we modify and streamline the earlier argument of \cite{Pat25} and instead of looking at maximum points of successive level sets, we look at the maximum points on successive spheres in a way that streamlines the near field argument. Note from \cref{lemma4,lemma5}, that $G'(a_k) > 0,G'(\bar{a}_k) $, and so we will have chosen $t_y$ small enough compared with $|G'(a_k)|$, $a_k$, and take the limit as $t_y \to 0$.

Thus, we have, at each step of the iteration, with $z_y$ the point on $\partial B_y$ where the minimum value of $G(\cdot)$ is attained on $\partial B_y$, we have,
\begin{equation}
G(z_y + t_y) = G(z_y) + t_y G'(z_y) + \frac{t_y^2}{2} G''(z_y) + O(t_y^{2+\alpha})
\end{equation}
where we have collected an error term $O(t_y^{2+\alpha})$ as in the previous case. We write $\bar{z}_{y+dy}$ for the point on the radial line joining the origin $0$ to the point $z_{y+dy}$, so that $G(z_y) = G(\bar{z}_{y+dy})$.

So, we get,
\begin{equation} \label{15}
\frac{G(z_y + t_y) - G(z_y)}{t_y} = G'(z_y) + \frac{t_y}{2} G''(z_y) + O(t_y^{1+\alpha})
\end{equation}

Note that $G(\bar{z}_{y+dy}) - G(z_{y+dy}) \ge G(z_y) - G(z_y + t_y) \ge 0$, and $w_y := |z_{y+dy} - \bar{z}_{y+dy}| \le t_y$. Thus, we get, using the Taylor expansion along the line joining $z_{y+dy}$ and $\bar{z}_{y+dy}$, that
\begin{equation} \label{16}
G(\bar{z}_{y+dy}) = G(z_{y+dy}) - w_y G'(z_{y+dy}) + \frac{w_y^2}{2} G''(z_{y+dy}) + O(w_y^{2+\alpha})
\end{equation}


\begin{figure}[h]
\centering
\includegraphics[width=0.65\textwidth]{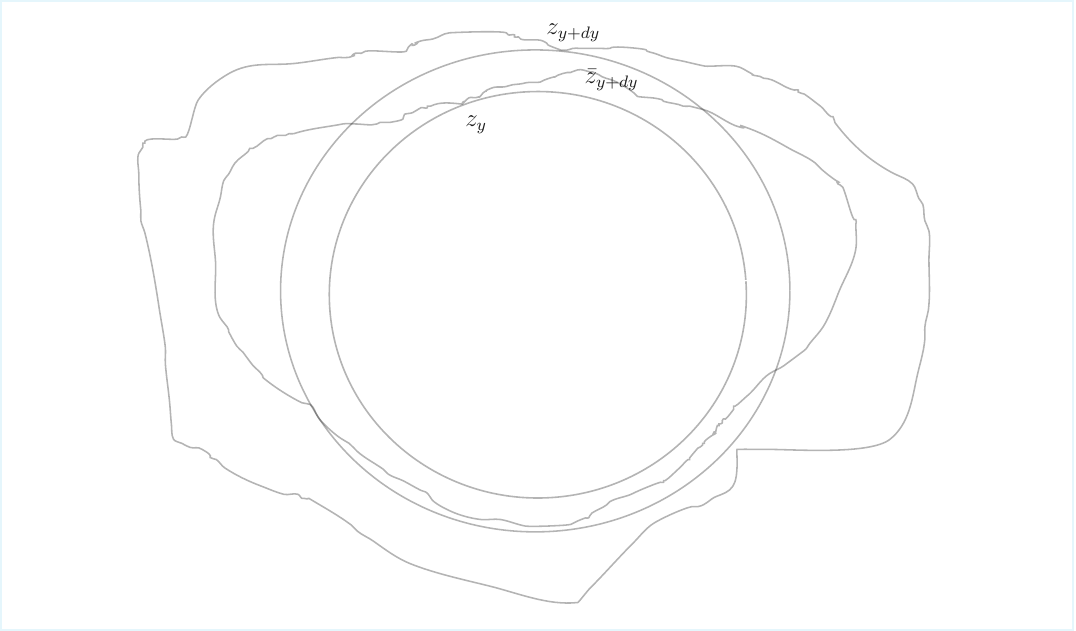}
\caption{The setting for the near field effect, showing the minimum points on two infinitesimally separated spheres, which are the points where the Green function is minimized on these infinitesimally separated spheres. We have $w_y=|z_{y+dy} -\bar{z}_{y+dy}|$, and the radial separation between the two infinitely separated spheres is $t_y$.}
\label{near field}
\end{figure}

Thus we have,
\begin{equation}
G(z_{y+dy}) - G(\bar{z}_{y+dy}) \le -w_y G'(z_{y+dy}) + \frac{w_y^2}{2} \left( -\left( \frac{n-1}{z_{y+dy}} \right) - \vec{B}(z_{y+dy}) \cdot \hat{z}_{y+dy} \right) G'(z_{y+dy}) + O(w_y^{2+\alpha})
\end{equation}
so we get, \begin{equation}G(z_{y+dy}) - G(\bar{z}_{y+dy}) \le w_y (-G'(z_{y+dy})) \left( 1 + \frac{(n-1)w_y}{2 z_{y+dy}} + \frac{1}{2} \vec{B}(z_{y+dy}) \cdot \hat{z}_{y+dy} w_y \right) + O(w_y^{2+\alpha})\end{equation}

so we get, \begin{equation}G'(z_{y+dy}) \le \frac{(-1)}{\left( 1 + w_y \left( \frac{n-1}{2 z_{y+dy}} + \frac{1}{2} \vec{B}(z_{y+dy}) \cdot \hat{z}_{y+dy} \right) \right)} \left( \frac{G(\bar{z}_{y+dy}) - G(z_{y+dy})}{w_y} + O(w_y^{1+\alpha}) \right)\end{equation}

so we get, \begin{equation}G'(z_{y+dy}) \le \left( 1 - w_y \left( \frac{n-1}{2 z_{y+dy}} + \frac{1}{2} \vec{B}(z_{y+dy}) \cdot \hat{z}_{y+dy} \right) \right) \left( \frac{G(z_{y+dy}) - G(\bar{z}_{y+dy})}{w_y} + O(w_y^{1+\alpha}) \right)\end{equation}

Noting that $w_y \le t_y$, we get from above, using that, $G(z_{y+dy}) - G(\bar{z}_{y+dy}) \le G(z_y + t_y) - G(z_y) \le 0$, that,
\begin{equation}
G'(z_{y+dy}) \le \left( 1 - w_y \left( \frac{n-1}{2 z_{y+dy}} + \frac{1}{2} \vec{B}(z_{y+dy}) \cdot \hat{z}_{y+dy} \right) \right) \left( \frac{G(z_y + t_y) - G(z_y)}{t_y} + O(w_y^{1+\alpha}) \right)
\end{equation}
Using Equation \ref{15}, we get,
\begin{equation}
G'(z_{y+dy}) \le \left( 1 - w_y \left( \frac{n-1}{2 z_{y+dy}} + \frac{1}{2} \vec{B}(z_{y+dy}) \cdot \hat{z}_{y+dy} \right) \right) \left( G'(z_y) + \frac{t_y}{2} G''(z_y) + O(t_y^{1+\alpha}) \right)
\end{equation}
so, 
\begin{multline}
G'(z_{y+dy}) \le \left( 1 - w_y \left( \frac{n-1}{2 z_{y+dy}} + \frac{1}{2} \vec{B}(z_{y+dy}) \cdot \hat{z}_{y+dy} \right) \right) \Bigg( G'(z_y) + \\ \frac{t_y}{2} \left( -\frac{n-1}{z_y} + \frac{1}{2} \vec{B}(z_y) \cdot \hat{z}_y \right) G'(z_y) + O(t_y^{1+\alpha}) \Bigg)
\end{multline}

so, \begin{multline}
G'(z_{y+dy}) \le G'(z_y) \left( 1 - w_y \left( \frac{n-1}{2 z_{y+dy}} + \frac{1}{2} \vec{B}(z_{y+dy}) \cdot \hat{z}_{y+dy} \right) - t_y \left( \frac{n-1}{2 z_y} + \frac{1}{2} \vec{B}(z_y) \cdot \hat{z}_y \right) \right) + O(t_y^{1+\alpha})
\end{multline}


We can again consider the Taylor expansion up to second order, to conclude
\begin{equation} \label{18}
|G'(z_y)| t_y + \frac{t_y^2}{2} G''(z_y + \theta t_y) \le |G'(z_{y+dy})| w_y + \frac{w_y^2}{2} G''(z_{y+dy} + \theta w_y)
\end{equation}
\begin{itemize}
    \item 

(a) First consider the case where we have
\begin{equation} \label{19}
(a) |G'(z_y)| w_y > t_y |G'(z_y)| - 2 t_y^2 \sup_{K_k} |G''(x)|
\end{equation}

Thus, we get that
\begin{multline} \label{20}
|G'(z_{y+dy})| \ge |G'(z_y)| - w_y \left( \frac{n-1}{2 z_{y+dy}} + \frac{1}{2} \vec{B}(z_{y+dy}) \cdot \hat{z}_{y+dy} \right) |G'(z_y)| \\
- t_y |G'(z_y)| \left( \frac{n-1}{2 z_y} + \frac{1}{2} \vec{B}(z_y) \cdot \hat{z}_y \right) + O(t_y^{1+\alpha})
\end{multline}

So when $$\left( \frac{n-1}{2 z_{y+dy}} + \frac{1}{2} \vec{B}(z_{y+dy}) \cdot \hat{z}_{y+dy} \right) \ge 0,$$ we simply use the fact that $w_y \le t_y$, to get from \ref{20}, that
\begin{equation} \label{21}
|G'(z_{y+dy})| \ge |G'(z_y)| - t_y |G'(z_y)| \left( \frac{n-1}{2 z_{y+dy}} + \frac{n-1}{2 z_y} + \frac{1}{2} \vec{B}(z_{y+dy}) \cdot \hat{z}_{y+dy} + \frac{1}{2} \vec{B}(z_y) \cdot \hat{z}_y \right) + O(t_y^{1+\alpha})
\end{equation}

 On the other hand, when $\left( \frac{n-1}{2 z_{y+dy}} + \frac{1}{2} \vec{B}(z_{y+dy}) \cdot \hat{z}_{y+dy} \right) < 0$, we get using equation \ref{19}, that
\begin{multline}
|G'(z_{y+dy})| \ge |G'(z_y)| - \left( \frac{n-1}{2 z_{y+dy}} + \frac{1}{2} \vec{B}(z_{y+dy}) \cdot \hat{z}_{y+dy} \right) \left( t_y |G'(z_y)| - 2 t_y^2 \sup_{K_k} |G''(x)| \right) \\
- t_y |G'(z_y)| \left( \frac{n-1}{2 z_y} + \frac{1}{2} \vec{B}(z_y) \cdot \hat{z}_y \right) + O(t_y^{1+\alpha})
\end{multline}

Thus,
\begin{multline} \label{eq71}
|G'(z_{y+dy})| \ge |G'(z_y)| - t_y |G'(z_y)| \Big( \frac{n-1}{2 z_{y+dy}} + \frac{n-1}{2 z_y} + \frac{1}{2} \vec{B}(z_{y+dy}) \cdot \hat{z}_{y+dy} \\+ \frac{1}{2} \vec{B}(z_y) \cdot \hat{z}_y \Big) + O(t_y^2) + O(t_y^{1+\alpha})
\end{multline}

\item (b) On the other hand, when we have the inequality opposite to Equation \ref{19} above, so that, $$|G'(z_y)| w_y \le |G'(z_y)| t_y - 2 t_y^2 \sup_{K_k} |G''(x)|,$$ then we use the Taylor expansion argument from Equation \ref{18} to get,

\begin{equation}
w_y |G'(z_{y+dy})| \ge |G'(z_y)| t_y + \frac{t_y^2}{2} G''(z_y + \theta t_y) - \frac{w_y^2}{2} G''(z_{y+dy} + \theta_1 w_y)
\end{equation}

so, \begin{equation}w_y |G'(z_{y+dy})| \ge |G'(z_y)| w_y + \frac{t_y^2}{2} G''(z_y + \theta t_y) - \frac{w_y^2}{2} G''(z_{y+dy} + \theta_1 w_y) + 2 t_y^2 \sup_{K_k} |G''(x)|.\end{equation}

so, \begin{equation}|G'(z_{y+dy})| \ge |G'(z_y)| + \frac{1}{w_y} \left( 2 t_y^2 \sup_{K_k} |G''(x)| + \frac{t_y^2}{2} G''(z_y + \theta t_y) - \frac{w_y^2}{2} G''(z_{y+dy} + \theta_1 w_y) \right).\end{equation}

so, \begin{equation} \label{eq75}
|G'(z_{y+dy})| \ge |G'(z_y)|
\end{equation}
\end{itemize}
Thus, combining the decay from Equation \cref{eq71,eq75} as well as Equation (23), we get that the gradient satisfies the solution
\begin{equation}
G'' + \left( \frac{n-1}{r} \right) G' + \vec{B} \cdot \hat{n} G' \le 0
\end{equation}

Note that \cref{eq71,eq75} are used over successive spheres that are infinitesimally separated by $t_y$, and one does not need to additionally rule out level sets that become arbitrarily close to each other in comparison to $t_y$, as is the case in \cite{Pat25}.

\subsection{Change in the Green's function in $K_{y_j}\setminus B(0,1/L)$.}

Now consider the three separate cases outlined in Section 4.1 .
\begin{itemize}
\item  Here we have $|G'(z_y)| > 1 \ \forall \ y \in (b_{i, k_0}, \frac{1}{2} R)$, and there is some $y_* = \max_{0 \le y \le 1} \{y : G'(s_y) \le 1\}$, so that $G'(s_{y_*}) \le 1$.

Thus, by the earlier computation, and a crude bound, we have for each $1\geq y\geq y_{*}$,
$$|G'(s_y)| \le \left| \frac{\partial G}{\partial r} \right|_{b_{i,k}} \frac{b_{i,k}^{n-1}}{(\frac{R}{2})^{n-1}}$$

Further, for $y \le y_*$, we have $|G'(s_y)| \le W$ from the argument preceding \cref{eq333}.

We also get for each $y_* \le y \le 1$, using the same argument, that,
\begin{equation} \label{eq27}
|G'(s_y)| \le W \cdot \frac{1}{(\frac{R}{2})^{n-1}} \left| \frac{\partial G}{\partial r} \right|_{b_{i,k}} b_{i,k}^{n-1}
\end{equation}

Thus, the total change in the value of the Green's function between the maximum point on $B(0, \frac{1}{L})$ and the maximum point of $\partial K_j$, for the fixed value of $j$ in $A_{i,j}$, is given by

$\Delta G_j = \Delta G_{1,j} + \Delta G_{2,j} + \Delta G_{3,j}$, where 
$\Delta G_{1,j} = (a_i^* - G(s_1))$, $\Delta G_{2,j} = G(s_1) - G(s_{y_*})$, $\Delta G_{3,j} = G(s_{y_*}) - G(s_y)$.

We thus have, \begin{align}\Delta G_{1,j} \le \left| \frac{\partial G}{\partial r} \right|_{b_{i,k}} b_{i,k}^{n-1} \left( L^{n-2} - \left( \frac{2}{R} \right)^{n-2} \right) \le \left| \frac{\partial G}{\partial r} \right|_{b_{i,k}} b_{i,k}^{n-1} L^{n-2}.\end{align}

Further, with the bound on $|G'(s_y)|$ from (\ref{eq27}), we get that $\Delta G_{2,j} \le W \left( \frac{2}{R} \right)^{n-1} \left| \frac{\partial G}{\partial r} \right|_{b_{i,k}} b_{i,k}^{n-1} \cdot D$ where $D$ is the diameter of convex body $K$. This is a crude bound that is attained by noting that we have chosen the points $s_y$ to be the maximum points on the boundaries $\partial K_y$, and thus at each stage, the change in the Green's function values is given by $|G'(s_y)| u_y + o(t_y^2)$, with $u_y \le R_0 t_y$ where $R_0$ is a fixed constant..

As $t_y \to 0$, the $o(t_y^2) M \approx o(t_y^2) \frac{D}{t_y} \approx o(t_y) \to 0$, and we have
$$\Delta G_{2,j} \le W \left( \frac{2}{R} \right)^{n-1} \left| \frac{\partial G}{\partial r} \right|_{b_{i,k}} b_{i,k}^{n-1} D$$

By a similar argument, for $\Delta G_{3,j}$, we see that the bound is given by $\Delta G_{3,j} \le WD$. \footnote{Note that this situation where $|G'(s_y)| \le W$ for $0 \le y \le y_*$, also takes care of the cases where the contribution due to the second derivative, at each infinitesimal step, to the change of the Green's function, is comparable to the contribution to the change of the Green's function due to the first order term. In particular, we may even have the situation where $|G'(s_y)| = 0$ in this regime.}. In this case, again, the total second order contribution is bounded crudely by $o(t_y^2) \frac{D}{t} = o(t_y)$ which goes to $0$ as $t_y \to 0$.

Thus the total change in the Green's function, $\Delta G_j$, is bounded from above by
\begin{equation*}
\Delta G_j \le \left| \frac{\partial G}{\partial r} \right|_{b_{i,k}} b_{i,k}^{n-1} L^{n-2} + W \left( \frac{2}{R} \right)^{n-1} \left| \frac{\partial G}{\partial r} \right|_{b_{i,k}} b_{i,k}^{n-1} D + WD.
\end{equation*}

As $j \to \infty$, we get, by the fact that $G_j \to 0$ continuously at the boundary, that
\begin{equation} \label{eq28}
\Delta G = \alpha_i^* \le \left| \frac{\partial G}{\partial r} \right|_{b_{i,k}} b_{i,k}^{n-1} \left( L^{n-2} + \left( \frac{2}{R} \right)^{n-1} WD \right) + WD
\end{equation}

\item In case 2, by similar arguments, we get that
\begin{equation*}
\Delta G_{1,j} \le \left| \frac{\partial G}{\partial r} \right|_{b_{i,k}} b_{i,k}^{n-1} \left( L^{n-2} - \left( \frac{1}{R_1} \right)^{n-2} \right) \le \left| \frac{\partial G}{\partial r} \right|_{b_{i,k}} b_{i,k}^{n-1} L^{n-2}
\end{equation*}
$$\Delta G_{2,j} \le WD,$$
where $\Delta G_{1,j} = \alpha_i^* - G(v_*)$, and $|v_*| = \frac{1}{R_1}$. 

By adopting the earlier argument, it also follows that for all $\frac{R}{2} \ge v \ge v_*$, we have $|G'(s_v)| \le W$. So we get, \begin{equation*}\Delta G = \alpha_i^* \le \Big|\frac{\partial G}{\partial r}\Big|_{b_{i,k}} b_{i,k}^{n-1} (L^{n-2}) + WD\end{equation*}

\item In this case, by an argument identical to ones used in the previous two cases, we get that $|G'(s_v)| \le W_1$ for some altered constant $W_1$, and so,
$$\alpha_i^* = \Delta G \le W_1 D.$$

\end{itemize}

Now, consider the two sequences of minimum points $a_{i,k}|_{k=1}^{\infty}$ and $\bar{a}_{i,k}|_{k=1}^{\infty}$ from \cref{lemma4,lemma5}. Also, consider the associated set of maximum points to $(a_{i,k}|_{k=1}^{\infty})$, which we denoted as $(b_{i,k}|_{k=1}^{\infty})$.

Thus, identical to Lemma 7 of \cite{Pat25}, we get the following result.

\begin{lemma} For any $a_{i,k_0} \le r_1 \le r_2$ (or any $\bar{a}_{i,k_0} \le r_1 < r_2$) with $r_1$ and $r_2$ minimum points of some spheres $\partial B_v$,
\begin{align}
G(r_1, 0) - G(r_2, 0) \ge C \left| \frac{\partial G}{\partial r} \right|_{a_{i,k_0}} a_{i,k_0}^{n-1} \Big( \frac{1}{r_1^{n-2}} - \frac{1}{r_2^{n-2}} \Big) \\
 G(r_1, 0) - G(r_2, 0)> C \left| \frac{\partial G}{\partial r} \right|_{\bar{a}_{i,k_0}} \bar{a}_{i,k_0}^{n-1}  \Big(\frac{1}{r_1^{n-2}} - \frac{1}{r_2^{n-2}}\Big).
\end{align}
\end{lemma}
\bigskip
Specifically we choose $a_{i,k_0} < r_1 = y_{i, \min} \le r_2 = \frac{1}{L}$, and get, as in \cite{Pat25}, that,
\begin{align}
2\alpha_i^* - \frac{1}{C_1} \alpha_i^* &> G(y_{i, \min}) - G(\frac{1}{L} s_{\min}) \\
&\ge C \left| \frac{\partial G}{\partial r} \right|_{b_{i,k}} a_{i,k}^{n-1} \left( \frac{1}{y_{i, \min}^{n-2}} - L^{n-2} \right) 
\end{align}
Further, we also have,
\begin{equation} \label{eq31}
2\alpha_{i}^* - \frac{1}{C_1} \alpha_{i}^{*}> C \left| \frac{\partial G}{\partial r} \right|_{\bar{a}_{i,k}} \bar{a}_{i,k}^{n-1} \left( \frac{1}{y_{i, \min}^{n-2}} - L^{n-2} \right)
\end{equation}

As $y_{i, \min} \to 0$ as $i \to \infty$, we choose the $k_0$ value appropriately so that $a_{i,k_0} < y_{i, \min}$, $\bar{a}_{i,k_0} < y_{i, \min}$, and so we will have
\begin{equation*}
\alpha_i^* \gtrsim C' \left| \frac{\partial G}{\partial r} \right|_{a_{i,k_0}} a_{i,k_0}^{n-1} \frac{1}{y_{i, \min}^{n-2}}, \quad \alpha_{i}^{*} \gtrsim C' \left| \frac{\partial G}{\partial r} \right|_{\bar{a}_{i,k_0}} \bar{a}_{i,k_0}^{n-1} \frac{1}{y_{i, \min}^{n-2}}
\end{equation*}

Here $C_1 \ge 1$ is a constant arising from the Harnack inequality.

Thus, comparing these above upper and lower estimates on $\alpha^{*}_i$, we get,

\begin{multline}
\frac{1}{2} \left( C' \left| \frac{\partial G}{\partial r} \right|_{a_{i,k_0}} a_{i,k_0}^{n-1} \frac{1}{y_{i, \min}^{n-2}} + C' \left| \frac{\partial G}{\partial r} \right|_{\bar{a}_{i,k_0}} \bar{a}_{i,k_0}^{n-1} \frac{1}{y_{i, \min}^{n-2}} \right)  \\ \le \alpha_{i}^{*} \le \left| \frac{\partial G}{\partial r} \right|_{b_{i,k}} b_{i,k}^{n-1} \left( L^{n-2} + \left( \frac{2}{R} \right)^{n-1} WD \right) + WD
\end{multline}

So, 
\begin{equation*}
\frac{1}{2} C' \left( \left| \frac{\partial G}{\partial r} \right|_{a_{i,k_0}} a_{i,k_0}^{n-1} + \left| \frac{\partial G}{\partial r} \right|_{\bar{a}_{i,k_0}} \bar{a}_{i,k_0}^{n-1} \right) \frac{1}{y_{i, \min}^{n-2}} \le \left| \frac{\partial G}{\partial r} \right|_{b_{i,k}} b_{i,k}^{n-1} \left( L^{n-2} + \left( \frac{2}{R} \right)^{n-1} WD \right) + WD
\end{equation*}

Now we will invoke \cref{lemma4,lemma5}, to get the desired contradiction. From \cref{lemma4}, we get that
\begin{equation*}
\frac{1}{2} C' \left| \frac{\partial G}{\partial r} \right|_{a_{i,k_0}} a_{i,k_0}^{n-1} > \frac{1}{M_0'} \left| \frac{\partial G}{\partial r} \right|_{b_{i,k}} b_{i,k}^{n-1}, 
\end{equation*}
 $\text{ and so as } y_{i, \min} \to 0$ we get $$\frac{1}{2} C' \left| \frac{\partial G}{\partial r} \right|_{a_{i,k_0}} a_{i,k_0}^{n-1} \frac{1}{y_{i, \min}^{n-2}} > \left| \frac{\partial G}{\partial r} \right|_{b_{i,k}} b_{i,k}^{n-1} \left( L^{n-2} + \Big(\frac{2}{R}\Big)^{n-1} WD \right)$$ as $i \to \infty$.

Similarly, from \cref{lemma5}, we get that for $i \to \infty$, we have
\begin{equation*}
\frac{1}{2} C' \left| \frac{\partial G}{\partial r} \right|_{\bar{a}_{i,k_0}} \bar{a}_{i,k_0}^{n-1} \frac{1}{y_{i, \min}^{n-2}} > W_1 D
\end{equation*}

Thus for $y_{i, \min}$ small enough, depending on the parameters $L$ (which in turn depends on $M$ and $D$), we get a contradiction, and thus the Case (1) is proved.

\section{Proof for Case 2}
Now we complete the proof for Case 2.
This argument follows from the argument in \cite{Pat25} with minor modifications, along with adopting the argument of Case (1) of this paper. In this case, we had defined $\tilde{\alpha}_N = \max_{S(0, R/2)} G(\cdot)$.

For each large enough integer $N$, we have $f(\alpha_N) > N f(\alpha_{2N})$ for some $\alpha_N > \tilde{\alpha}_N$. Thus there exist boundary points $z_N \in \partial \Omega_N, z_{2N} \in \partial \Omega_{2N}$ with $z_{2N}$ being the point at the minimum distance to $\partial \Omega_{2N}$ from the origin, and $z_N$ the point at the maximum distance to $\partial \Omega_N$ from the origin, and by definition $z_N \in B(0, R/2)$.
As $N \to \infty$, we have $\frac{|z_N|}{|z_{2N}|} \to \infty$. Consider the value $\alpha^* = \max_{x \in S(0, |z_N|)} G(x) = G(z_N, 0)$.

Now as in the Case (1), we have sequences $\{a_{N,k}\}_{k=1}^{\infty}, \{\bar{a}_{N,k}\}_{k=1}^{\infty}, \{b_{N,k}\}_{k=1}^{\infty}$ satisfying the conditions of \cref{lemma4,lemma5}, so that we have for some constant $L_1>0$,
\begin{equation} \label{eq33}
\alpha_N^* \le C' \left| \frac{\partial G}{\partial r} \right|_{b_{N,k}} b_{N,k}^{n-1} \left( \frac{1}{|z_N|^{n-2}} + \left( \frac{2}{R} \right)^{n-1} L_1 D \right) + L_1 D
\end{equation}

Further,
\begin{equation} \label{eq34}
\alpha_N^* \gtrsim C \left| \frac{\partial G}{\partial r} \right|_{a_{N,k}} a_{N,k}^{n-1} \left( \frac{1}{|z_{2N}|^{n-2}} - \frac{1}{|z_N|^{n-2}} \right)
\end{equation}
\begin{equation} \label{eq35}
\bar{\alpha}_N^* \gtrsim C \left| \frac{\partial G}{\partial r} \right|_{\bar{a}_{N,k}} \bar{a}_{N,k}^{n-1} \left( \frac{1}{|z_{2N}|^{n-2}} - \frac{1}{|z_N|^{n-2}} \right)
\end{equation}

Thus, as $|z_{2N}|/|z_N| \to 0$ as $N \to \infty$, \label{eq33,eq34,eq35} give us a contradiction. In the case where instead of \cref{eq33}, we have..

\begin{equation*}
\alpha_N^* \le \left| \frac{\partial G}{\partial r} \right|_{b_{N,k}} b_{N,k}^{n-1} \left( \frac{1}{|z_N|^{n-2}} \right) + L_1 D
\end{equation*}
or when $\alpha_N^* \le L_1 D$, the argument also adopts immediately and we have the result.

This completes the proof of \cref{thm2}.\quad
\vspace{1em}

\textbf{Remark:} Note that our construction in Subsection 2 gives an alternate argument for the result of Theorem 3 of \cite{Pat25}. In particular, as the pole of the Green's function is moved arbitrarily close to the boundary $\partial K$, the $r_y$ parameters of Subsection 4.1 have worse bounds, leading to a worse $L_1$ parameter, as the pole is taken arbitrarily close to the boundary, which leads to admissible $y_{i,\min}$ values to be arbitrarily small, leading to worse point-wise upper bounds through Equations 41 and 42 of \cite{Pat25}.

\section{Acknowledgements:} The author thanks Debanjan Nandi for helpful discussions.

\section{Use of LLMs} No LLM's have been used to generate any of the proofs in this manuscript.

\bigskip

Email: ap7mx@missouri.edu.

\end{document}